\documentclass[11pt]{article}
\usepackage[a4paper,margin=1in]{geometry}
\usepackage{amsmath,amssymb,amsthm,mathtools}
\usepackage{microtype}
\usepackage{hyperref}
\usepackage{aliascnt}
\usepackage[nameinlink,capitalise,noabbrev]{cleveref}
\usepackage{lmodern}

\hypersetup{%
  colorlinks=true,
  linkcolor=blue,
  citecolor=blue,
  urlcolor=blue,
  pdftitle={The Faraday Form of Conformal Products: Weyl Curvature and Four-Dimensional Einstein Rigidity},
  pdfauthor={Jae Won Lee}
}
\allowdisplaybreaks

\newtheorem{theorem}{Theorem}[section]
\newaliascnt{lemma}{theorem}
\newtheorem{lemma}[lemma]{Lemma}
\aliascntresetthe{lemma}
\newaliascnt{proposition}{theorem}
\newtheorem{proposition}[proposition]{Proposition}
\aliascntresetthe{proposition}
\newaliascnt{corollary}{theorem}
\newtheorem{corollary}[corollary]{Corollary}
\aliascntresetthe{corollary}
\newaliascnt{remark}{theorem}
\newtheorem{remark}[remark]{Remark}
\aliascntresetthe{remark}
\newaliascnt{example}{theorem}
\newtheorem{example}[example]{Example}
\aliascntresetthe{example}
\newaliascnt{definition}{theorem}

\aliascntresetthe{definition}
\newaliascnt{question}{theorem}
\newtheorem{question}[question]{Question}
\aliascntresetthe{question}

\newcommand{\Ric}{\operatorname{Ric}}
\newcommand{\tr}{\operatorname{tr}}
\newcommand{\rk}{\operatorname{rank}}
\newcommand{\Id}{\operatorname{Id}}
\newcommand{\KW}{\mathcal K_W}
\newcommand{\HS}{\mathrm{HS}}
\newcommand{\R}{\mathbb R}
\newcommand{\Hy}{\mathbb H}
\newcommand{\Scal}{\operatorname{Scal}}
\newcommand{\Hess}{\operatorname{Hess}}
\newcommand{\Rish}{\operatorname{Ric}^\sharp}

\title{\textbf{The Faraday Form of Conformal Products: Weyl Curvature and Four-Dimensional Einstein Rigidity}}
\author{Jae Won Lee\\
Department of Mathematics Education, Gyeongsang National University\\
Jinju 52828, Republic of Korea\\
\texttt{leejaew@gnu.ac.kr}}
\date{}

\begin{document}
\maketitle

\begin{abstract}
Let $(M^n,g)$ carry a conformal product structure whose adapted Weyl connection $D$ preserves orthogonal distributions of ranks $p,q\ge2$. We express the Faraday form $d\theta$ directly in terms of the Weyl curvature of $g$. If $S$ is the orthogonal involution determined by the splitting and
\[
\mathcal K_W(X)=\sum_i W_{X,e_i}(Se_i)
\]
for any local orthonormal frame, then
\[
(d\theta)^\sharp
=
\frac{n-2}{4(p-1)(q-1)}[\mathcal K_W,S].
\]
No Ricci-curvature assumption is required. The same curvature calculation gives a companion formula for the symmetric part of $\nabla\theta$,
\[
S\Theta^s-\Theta^sS+\theta^\sharp\wedge S\theta^\sharp
=
\frac{p-q}{n-2}(d\theta)^\sharp
+\frac1{n-2}[S,\Ric^\sharp],
\qquad \Theta=\nabla\theta.
\]
When the Ricci tensor is block diagonal with respect to the product splitting, these formulas lead to explicit identities for the two components of the Lee form. In dimension four, the rank-$(2,2)$ splitting determines an ambi-Hermitian pair. The classical ambi-Hermitian curvature identities imply local closedness when the Ricci tensor is invariant under both complex structures; we also recover this conclusion from the mixed Weyl trace above. Finally, if $(M^4,g)$ is compact and Einstein, no specialness assumption is needed: $D=\nabla^g$. Hence the universal cover is the product of two simply connected surfaces with the same constant Gaussian curvature.
\end{abstract}



\section{Introduction}

A conformal product may be viewed as a conformal manifold endowed with a Weyl connection of reducible holonomy. More precisely, the adapted Weyl connection $D$ preserves an orthogonal splitting
\[
TM=T_1\oplus T_2.
\]
Belgun and Moroianu established the local theory of conformal products and showed that the Faraday form of $D$ has only mixed components with respect to this splitting \cite{BelgunMoroianu2011}. Faraday curvature is a standard invariant in Weyl geometry \cite{Calderbank2001}. The question considered here is how the non-closedness of the adapted Weyl connection is seen by the Riemannian curvature of a chosen metric $g$ in the conformal class.

Conformally Einstein product metrics have been studied from several directions. K\"uhnel and Rademacher described the local alternatives for product metrics conformal to Einstein metrics, including the exceptional four-dimensional branch in which both factors are extremal surfaces \cite[Theorem~3.2]{KuhnelRademacher2016}. Moroianu and Pilca obtained mixed Ricci identities for conformal products and proved a compact Einstein reduction under an additional specialness assumption \cite{MoroianuPilca2024,MoroianuPilca2026}. In the constant-sectional-curvature case, Jiang used the product involution to derive a two-by-two system that forces the adapted Weyl connection to be closed when both factor ranks are at least two \cite{Jiang2026}.

Let
\[
S|_{T_1}=\Id,\qquad S|_{T_2}=-\Id,
\]
and define
\[
\mathcal K_W(X)=\sum_i W_{X,e_i}(Se_i)
\]
for any local orthonormal frame $\{e_i\}$. We write $\mathcal K_W$ for this $S$-weighted contraction of the Weyl tensor. Our first result is the identity
\begin{equation}\label{eq:intro-main}
(d\theta)^\sharp
=
\frac{n-2}{4(p-1)(q-1)}[\mathcal K_W,S],
\end{equation}
valid whenever $p=\rk T_1\ge2$ and $q=\rk T_2\ge2$. No assumption on the Ricci tensor is needed. Thus $D$ is closed if and only if $\mathcal K_W$ preserves the two summands.

For $X\in T_1$ and $U\in T_2$, \eqref{eq:intro-main} is equivalent to
\begin{equation}\label{eq:intro-mixed-Weyl}
d\theta(X,U)
=
\frac{n-2}{(p-1)(q-1)}
\sum_{a=1}^{p}W(X,e_a,e_a,U).
\end{equation}
The local curvature formulas of Moroianu--Pilca already contain the partial traces from which this scalar identity can be recovered after the Weyl correction is inserted \cite{MoroianuPilca2024}. Accordingly, the point of \eqref{eq:intro-main} is not the local partial trace by itself. The formula gives an invariant endomorphism description of the obstruction and makes the cancellation of the mixed Ricci terms transparent.

The same calculation also gives information about the symmetric part of $\nabla\theta$. If $\Theta=\nabla\theta$, then
\begin{equation}\label{eq:intro-symmetric-companion}
S\Theta^s-\Theta^sS+\theta^\sharp\wedge S\theta^\sharp
=
\frac{p-q}{n-2}(d\theta)^\sharp
+\frac1{n-2}[S,\Ric^\sharp].
\end{equation}
This identity may also be recovered from local double-twisted product formulas. The intrinsic derivation is useful because \eqref{eq:intro-main} and \eqref{eq:intro-symmetric-companion} arise from the same two-by-two curvature system. Under the block-diagonal condition $\Ric(T_1,T_2)=0$, the Ricci commutator vanishes and, for $X\in T_1$ and $U\in T_2$,
\begin{align*}
(\nabla_X\theta)(U)
&=\theta_1(X)\theta_2(U)+\frac{q-1}{n-2}d\theta(X,U),\\
(\nabla_U\theta)(X)
&=\theta_1(X)\theta_2(U)-\frac{p-1}{n-2}d\theta(X,U).
\end{align*}

Section~3 develops the consequences of the block-diagonal Ricci condition. Writing $\theta=\theta_1+\theta_2$, we obtain
\[
d\theta_1
=
\frac{p-1}{n-2}d\theta+\theta_1\wedge\theta_2,
\qquad
d\theta_2
=
\frac{q-1}{n-2}d\theta-\theta_1\wedge\theta_2.
\]
These formulas distinguish closedness of the Weyl connection from specialness of the metric in the sense of \cite{MoroianuPilca2026}. A local conformally flat Einstein example has $d\theta=0$ although neither Lee component is closed. Another local example has block-diagonal Ricci tensor but $d\theta\ne0$, showing that the block-diagonal Ricci condition alone does not imply closedness. We do not produce a non-closed Einstein example; the higher-dimensional existence question is left open in Section~5.

In dimension four, a rank-$(2,2)$ splitting gives two orthogonal complex structures $J_+$ and $J_-$ of opposite orientations, both preserved by $D$, and $S=-J_+J_-$. This is an ambi-Hermitian pair in the standard terminology \cite[Definition~1 and Lemma~1]{ApostolovCalderbankGauduchon2016}. In the Lee-form convention used both here and in \cite{ApostolovCalderbankGauduchon2016}, the two canonical Hermitian Weyl connections coincide with $D$, so the two Hermitian Lee forms are both equal to $\theta$. If the Ricci tensor is invariant under both structures, classical ambi-Hermitian curvature theory implies $d\theta=0$ \cite[Proposition~1 and Section~3.3]{ApostolovCalderbankGauduchon2016}. We then record how the same half-Weyl degeneracy conditions annihilate the mixed contraction in \eqref{eq:intro-mixed-Weyl}; this is a consistency check between the classical Hermitian description and the curvature formula \eqref{eq:intro-main}, rather than an independent proof of closedness.

The principal four-dimensional result is global. If $(M^4,g)$ is compact and Einstein and carries a rank-$(2,2)$ conformal product structure, then
\[
D=\nabla^g.
\]
Compared with \cite[Theorem~1.1]{MoroianuPilca2026}, this strengthens the conclusion in two respects in rank $(2,2)$: the specialness hypothesis is removed, and the warped-product alternative on the universal cover collapses to an actual Riemannian product. Closedness alone does not imply specialness: when $p=q=2$ and $d\theta=0$,
\[
d\theta_1=\theta_1\wedge\theta_2,
\qquad
d\theta_2=-\theta_1\wedge\theta_2.
\]
The local conformally flat example in Section~2 shows that the right-hand side need not vanish.

The compact proof uses several established global results, each at a separate stage. Belgun--Flamencourt--Moroianu exclude the non-flat, non-exact LCP branch for compact Einstein metrics \cite[Definition~4.1 and Theorem~4.5]{BelgunFlamencourtMoroianu2025}, while the flat non-exact similarity branch is treated separately using Fried's classification \cite{Fried1980}. Exactness then produces a complete product metric on the universal cover. To reduce the conformal factor to one de Rham factor we use the full four-dimensional alternative in K\"uhnel--Rademacher \cite[Theorem~3.2]{KuhnelRademacher2016}, including its extremal-surface branch; this point is essential when both factors have dimension two. Tashiro's classification of complete manifolds with a concircular scalar field is used both in excluding that extremal two-sided branch and in making the remaining warped-product base compact \cite[Theorem~1]{Tashiro1965}. The final two-dimensional ODE is the rank-$(2,2)$ specialization of the warped-product equations in \cite[Corollary~5.4, Proposition~5.7 and Lemma~5.8]{KuhnelRademacher2016}. The universal cover consequently splits as
\[
(\widetilde M,\widetilde g)
\cong
\Sigma_\lambda^2\times\Sigma_\lambda^2.
\]

Section~2 proves the curvature formula for $d\theta$ and the companion identity for the symmetric part of $\nabla\theta$. Section~3 treats the components of the Lee form under block-diagonal Ricci curvature. Section~4 recalls the ambi-Hermitian local closedness mechanism, relates it to the mixed Weyl contraction, and proves the compact Einstein rigidity theorem. Section~5 states the remaining higher-dimensional problem. Appendix~A contains an independent local calculation of the mixed Ricci tensor and the relevant partial traces.

\section{Weyl curvature and the Faraday form}

There are two useful ways to obtain the formula for the Faraday form. Appendix~A starts from local double-twisted product formulas and computes the relevant partial Weyl trace. Here we use the product involution $S$ instead. The resulting calculation is independent of a product chart and places the skew and symmetric parts of $\nabla\theta$ in the same two-by-two curvature system. It also shows directly why the Ricci commutator disappears from the final formula for $d\theta$.

We begin by fixing conventions. A Weyl connection $D$ on $(M,g)$ is written in terms of its Lee form $\theta$ as
\begin{equation}\label{eq:weyl-connection}
D_XY
=
\nabla_XY+\theta(Y)X+\theta(X)Y-g(X,Y)\theta^\sharp.
\end{equation}
Equivalently,
\[
Dg=-2\theta\otimes g.
\]
If $\theta=d\varphi$, then $D$ is the Levi-Civita connection of $e^{2\varphi}g$.

Our curvature convention is
\begin{equation}\label{eq:curvature-convention}
R(X,Y)Z
=
\nabla_X\nabla_YZ-
\nabla_Y\nabla_XZ-
\nabla_{[X,Y]}Z.
\end{equation}
Thus a metric of constant sectional curvature $\kappa$ satisfies
\[
R(X,Y)Z
=
\kappa\bigl(g(Y,Z)X-g(X,Z)Y\bigr).
\]
For a two-form $\omega$, the corresponding skew endomorphism $\omega^\sharp$ is determined by
\begin{equation}\label{eq:sharp-convention}
g(\omega^\sharp X,Y)=\omega(X,Y).
\end{equation}
We use the two-form norm convention
\[
|\omega|^2=\sum_{i<j}\omega(e_i,e_j)^2,
\]
so that $|\omega^\sharp|_{\HS}^2=2|\omega|^2$.
For vectors $u,v$, we use the skew endomorphism
\[
(u\wedge v)(X)=g(v,X)u-g(u,X)v.
\]
The Ricci tensor is $\Ric(X,Y)=\sum_ig(R(X,e_i)e_i,Y)$, and $\Rish$ denotes
the associated self-adjoint endomorphism,
\[
g(\Rish X,Y)=\Ric(X,Y),
\qquad
\Scal_g=\tr\Rish.
\]
Finally, $[A,B]=AB-BA$ denotes the commutator.

Let $D$ be the adapted Weyl connection of a conformal product with orthogonal $D$-parallel splitting
\[
TM=T_1\oplus T_2,
\qquad p=\rk T_1,\quad q=\rk T_2,
\qquad n=p+q.
\]
Define the orthogonal involution
\begin{equation}
S|_{T_1}=\Id,
\qquad
S|_{T_2}=-\Id,
\qquad
s:=\tr S=p-q.
\end{equation}
The standard conformal-product identity is
\begin{equation}\label{eq:nablaS}
\nabla_XS
=
SX\odot\theta^\sharp-S\theta^\sharp\odot X,
\end{equation}
where $u\odot v$ denotes the symmetric endomorphism $X\mapsto g(u,X)v+g(v,X)u$. Formula \eqref{eq:nablaS} is equivalent to the fact that $D$ preserves $S$; see, for example, \cite{BelgunMoroianu2011,MoroianuPilca2026}. We give all curvature calculations used below explicitly; \cite{Jiang2026} is cited only for comparison with the constant-curvature two-by-two system.

\begin{lemma}
Let $W$ be the Weyl tensor of $g$ and define
\begin{equation}\label{eq:KW-def}
\KW(X)=\sum_{i=1}^n W_{X,e_i}(Se_i),
\end{equation}
for any local orthonormal frame $\{e_i\}$. Then $\KW$ is frame-independent and self-adjoint.
\end{lemma}

\begin{proof}
Frame-independence follows because \eqref{eq:KW-def} is the contraction of the tensor $W$ with the endomorphism $S$ in the two middle slots. For the symmetry statement, choose an orthonormal frame adapted to the splitting, so
\[
Se_i=\varepsilon_i e_i,
\qquad \varepsilon_i\in\{1,-1\}.
\]
Then
\[
\KW(X)=\sum_i\varepsilon_i W_{X,e_i}e_i.
\]
For fixed $e_i$, set $J^W_{e_i}X=W_{X,e_i}e_i$. By the standard symmetries of $W$,
\begin{align*}
g(J^W_{e_i}X,Y)
&=W(X,e_i,e_i,Y)\\
&=W(Y,e_i,e_i,X)\\
&=g(X,J^W_{e_i}Y).
\end{align*}
Thus each $J^W_{e_i}$ is self-adjoint, and hence so is their signed sum $\KW$.
\end{proof}

\begin{example}[Einstein product model]
Let $(M_1^p,g_1)$ and $(M_2^q,g_2)$ be Einstein manifolds with the same Einstein constant $\lambda$, and take
\[
(M,g)=(M_1\times M_2,g_1+g_2).
\]
Then $g$ is Einstein and the product splitting is $\nabla$-parallel. Thus $\theta=0$, $S$ is parallel, and the curvature preserves both factors. Consequently $\KW$ is block diagonal and $[\KW,S]=0$. This is the basic equality case for the theorem below.
\end{example}

Set
\[
\Theta:=\nabla\theta,
\qquad
g(\Theta X,Y)=(\nabla_X\theta)(Y),
\]
and split
\[
\Theta^s=\frac12(\Theta+\Theta^t),
\qquad
\Theta^a=\frac12(\Theta-\Theta^t).
\]
Then
\begin{equation}
(d\theta)^\sharp=2\Theta^a.
\end{equation}
Motivated by the decomposition used in \cite{Jiang2026}, and using the skew-endomorphism convention fixed above, define
\begin{equation}
P=S\Theta^s-\Theta^sS+\theta^\sharp\wedge S\theta^\sharp,
\qquad
Q=S\Theta^aS-\Theta^a.
\end{equation}
The sign of the quadratic term in $P$ is tied to our convention
$(u\wedge v)(X)=g(v,X)u-g(u,X)v$.

\begin{lemma}[The curvature system]\label{lem:PQ-system}
For every conformal product, without any Ricci-curvature assumption,
\begin{equation}\label{eq:PQ1}
(n-2)P+sQ=[S,\Ric^\sharp],
\end{equation}
and
\begin{equation}\label{eq:PQ2}
sP+(n-2)Q
=[S,\KW]+\frac{s}{n-2}[S,\Ric^\sharp].
\end{equation}
No rank assumption on $T_1$ and $T_2$ is used here; the ranks enter only when the system is solved in \cref{thm:main-obstruction}.
\end{lemma}

\begin{proof}
Differentiating \eqref{eq:nablaS} and antisymmetrizing gives the Ricci identity
\[
R_{X,Y}S:=[R_{X,Y},S],
\]
and explicitly
\begin{align}
R_{X,Y}S={}&
SY\odot \Theta X-SX\odot \Theta Y
+S\Theta Y\odot X-S\Theta X\odot Y \notag\\
&+\theta(Y)
(SX\odot\theta^\sharp-S\theta^\sharp\odot X) \notag\\
&+\theta(X)
(S\theta^\sharp\odot Y-SY\odot\theta^\sharp) \notag\\
&-|\theta|^2(SX\odot Y-SY\odot X).
\label{eq:full-curvature-S}
\end{align}
This identity follows directly from $DS=0$ and is also used in
\cite{MoroianuPilcaReducible2025,Jiang2026}.

Let
\[
\mathcal A:=\Rish,
\qquad
\sigma:=\Scal_g.
\]
With our curvature convention the Weyl decomposition is
\begin{align}
R_{X,Y}Z={}&W_{X,Y}Z
+\frac1{n-2}\bigl(
\Ric(Y,Z)X-\Ric(X,Z)Y \notag\\
&\hspace{28mm}
+g(Y,Z)\mathcal AX-g(X,Z)\mathcal AY
\bigr) \notag\\
&-\frac{\sigma}{(n-1)(n-2)}
\bigl(g(Y,Z)X-g(X,Z)Y\bigr).
\label{eq:weyl-decomp-general}
\end{align}
Define
\[
\mathcal C X:=\sum_i(R_{X,e_i}S)e_i.
\]
Tracing the Weyl decomposition gives
\begin{equation}
\mathcal C=\KW+\mathcal E,
\end{equation}
where
\begin{equation}\label{eq:E-Ricci}
\mathcal E
=
\frac1{n-2}
\left(
\tr(\mathcal AS)\Id+s\mathcal A-\mathcal AS-(n-1)S\mathcal A
\right)
-\frac{\sigma}{(n-1)(n-2)}(s\Id-S).
\end{equation}
The first curvature trace is $\sum_iR_{X,e_i}(Se_i)$, whereas the second part of $(R_{X,e_i}S)e_i$ contributes $-S\mathcal AX$, since $\sum_iR_{X,e_i}e_i=\mathcal AX$.

We next compute the same trace from \eqref{eq:full-curvature-S}. Put $\xi:=\theta^\sharp$, and write $u\otimes v$ for the endomorphism $X\mapsto g(v,X)u$. Direct contraction gives
\begin{align}
\mathcal C={}&
s\Theta-\Theta S-(\tr \Theta)S+\tr(S\Theta)\Id-(n-1)S\Theta \notag\\
&+\xi\otimes S\xi+(n-1)S\xi\otimes\xi-s\xi\otimes\xi \notag\\
&-(n-1)|\xi|^2S+
\bigl(s|\xi|^2-\langle S\xi,\xi\rangle\bigr)\Id.
\label{eq:traced-C}
\end{align}

For any endomorphism $B$, set
\[
B^{\rm sk}:=\frac12(B-B^t),
\qquad
\Pi_-(B):=\frac12(B-SBS).
\]
The elementary pieces entering the projection are
\[
\Pi_-(S\Theta^s-\Theta^sS)=S\Theta^s-\Theta^sS,
\qquad
\Pi_-(\Theta^a)=-\frac12Q,
\]
\[
\Pi_-(S\Theta^aS)=\frac12Q,
\qquad
\Pi_-(\xi\wedge S\xi)=\xi\wedge S\xi.
\]
Substituting these terms in \eqref{eq:traced-C} yields
\begin{align}
2\Pi_-(\mathcal C^{\rm sk})
&=-(n-2)\bigl(S\Theta^s-\Theta^sS+\xi\wedge S\xi\bigr)
-s\bigl(S\Theta^aS-\Theta^a\bigr) \notag\\
&=-\bigl((n-2)P+sQ\bigr).
\end{align}
On the other hand, $\KW$ is self-adjoint. From \eqref{eq:E-Ricci}, using that both $\mathcal A$ and $S$ are self-adjoint,
\[
\mathcal E^{\rm sk}=\frac12[\mathcal A,S].
\]
The commutator $[\mathcal A,S]$ anti-commutes with $S$, hence survives $\Pi_-$. Comparing the two expressions for $\mathcal C$ gives
\[
-\bigl((n-2)P+sQ\bigr)=[\mathcal A,S],
\]
which is \eqref{eq:PQ1}.

Multiplying $\mathcal C=\KW+\mathcal E$ on the left by $S$ and projecting once more gives
\begin{align}
2\Pi_-\bigl((S\mathcal C)^{\rm sk}\bigr)
&=
s\bigl(S\Theta^s-\Theta^sS+\xi\wedge S\xi\bigr)
+(n-2)\bigl(S\Theta^aS-\Theta^a\bigr) \notag\\
&=sP+(n-2)Q.
\end{align}
Since $\KW$ is self-adjoint,
\[
2(S\KW)^{\rm sk}=S\KW-\KW S=[S,\KW].
\]
A direct multiplication of \eqref{eq:E-Ricci} by $S$ gives
\[
(S\mathcal E)^{\rm sk}
=\frac{s}{2(n-2)}[S,\mathcal A].
\]
Both commutators anti-commute with $S$, so the $\Pi_-$ projection leaves them unchanged. This proves \eqref{eq:PQ2}.

If $[\Ric^\sharp,S]=0$, both Ricci commutators vanish and the system reduces to the homogeneous first equation and the Weyl-forced second equation used in the block-diagonal Ricci specialization below. If $g$ is Einstein, this reduction is automatic.
\end{proof}

\begin{remark}
If $g$ has constant sectional curvature, then $W=0$ and $\Ric^\sharp$ is a scalar multiple of the identity. Both commutators in \cref{lem:PQ-system} therefore vanish, and one recovers the homogeneous two-by-two system appearing in Jiang's argument for the case in which both ranks are at least two, after translating conventions. The general system above shows that the Weyl commutator is the only term that survives in the Faraday equation after the Ricci terms cancel.
\end{remark}

\begin{example}[Constant-curvature check]
If $g$ has constant sectional curvature, then $W=0$ and hence $\KW=0$. The system in \cref{lem:PQ-system} becomes
\[
(n-2)P+sQ=0,\qquad sP+(n-2)Q=0.
\]
For $p,q\ge2$ its determinant is $4(p-1)(q-1)>0$, so $P=Q=0$. Since $Q=-(d\theta)^\sharp$ for a conformal product, one recovers $d\theta=0$, exactly as in the constant-curvature case with both ranks at least two of \cite{Jiang2026}.
\end{example}

\begin{theorem}[Curvature formula for the Faraday form]\label{thm:main-obstruction}
Let $(M^n,g)$ carry a conformal product structure with orthogonal $D$-parallel splitting $TM=T_1\oplus T_2$, where
\[
p=\rk T_1\ge2,
\qquad
q=\rk T_2\ge2.
\]
Then, without any Ricci-curvature assumption,
\begin{equation}\label{eq:main-obstruction}
\boxed{
(d\theta)^\sharp
=
\frac{n-2}{4(p-1)(q-1)}[\KW,S].
}
\end{equation}
Moreover,
\begin{equation}\label{eq:symmetric-companion}
\boxed{
S\Theta^s-\Theta^sS+\theta^\sharp\wedge S\theta^\sharp
=
\frac{p-q}{n-2}(d\theta)^\sharp
+\frac1{n-2}[S,\Ric^\sharp].
}
\end{equation}
In particular, under the additional block-diagonal Ricci condition $[\Ric^\sharp,S]=0$, the last correction term vanishes.
\end{theorem}

\begin{proof}
The system \eqref{eq:PQ1}--\eqref{eq:PQ2} is
\[
\begin{pmatrix}
n-2&s\\
s&n-2
\end{pmatrix}
\begin{pmatrix}
P\\Q
\end{pmatrix}
=
\begin{pmatrix}
[S,\Ric^\sharp]\\[1mm]
[S,\KW]+\dfrac{s}{n-2}[S,\Ric^\sharp]
\end{pmatrix}.
\]
Its determinant is
\[
\Delta=(n-2)^2-s^2=4(p-1)(q-1)>0.
\]
Solving for $Q$, the Ricci terms cancel exactly:
\begin{align*}
\Delta Q
&=-s[S,\Ric^\sharp]
+(n-2)\left([S,\KW]+\frac{s}{n-2}[S,\Ric^\sharp]\right)\\
&=(n-2)[S,\KW].
\end{align*}

For every conformal product, if $\theta=\theta_1+\theta_2$ according to $T_1^*\oplus T_2^*$, then
\[
d_1\theta_1=0,
\qquad
d_2\theta_2=0;
\]
see \cite[Lemma~4.6]{BelgunMoroianu2011} and \cite[Remark~2.4]{MoroianuPilca2026}. Consequently $d\theta$ is of mixed type, so
\[
S(d\theta)^\sharp S=-(d\theta)^\sharp.
\]
Since $(d\theta)^\sharp=2\Theta^a$,
\[
Q=S\Theta^aS-\Theta^a=-(d\theta)^\sharp.
\]
Thus
\[
(d\theta)^\sharp
=-\frac{n-2}{\Delta}[S,\KW]
=\frac{n-2}{\Delta}[\KW,S],
\]
which proves \eqref{eq:main-obstruction}.

Solving the first equation for $P$ now gives
\[
P=\frac{s}{n-2}(d\theta)^\sharp+\frac1{n-2}[S,\Ric^\sharp].
\]
Since $s=p-q$, this is exactly \eqref{eq:symmetric-companion}.
\end{proof}

\begin{corollary}[Mixed covariant derivative of the Lee form]\label{cor:mixed-nabla-theta}
For $X\in T_1$ and $U\in T_2$ one has, without any Ricci assumption,
\begin{align}
(\nabla_X\theta)(U)
&=\theta_1(X)\theta_2(U)
+\frac{q-1}{n-2}d\theta(X,U)
+\frac1{n-2}\Ric(X,U),
\\
(\nabla_U\theta)(X)
&=\theta_1(X)\theta_2(U)
-\frac{p-1}{n-2}d\theta(X,U)
+\frac1{n-2}\Ric(X,U).
\end{align}
If $\Ric(T_1,T_2)=0$, these reduce to
\begin{align}
(\nabla_X\theta)(U)
&=\theta_1(X)\theta_2(U)+\frac{q-1}{n-2}d\theta(X,U),
\\
(\nabla_U\theta)(X)
&=\theta_1(X)\theta_2(U)-\frac{p-1}{n-2}d\theta(X,U).
\end{align}
\end{corollary}

\begin{proof}
Write $\xi=\theta^\sharp$. For $X\in T_1$ and $U\in T_2$,
\[
g\bigl((S\Theta^s-\Theta^sS)X,U\bigr)=-2\Theta^s(X,U),
\]
and
\[
g\bigl((\xi\wedge S\xi)X,U\bigr)=2\theta_1(X)\theta_2(U).
\]
Moreover,
\[
g([S,\Ric^\sharp]X,U)=-2\Ric(X,U).
\]
Taking the $(X,U)$ component of \eqref{eq:symmetric-companion} gives
\[
\Theta^s(X,U)
=\theta_1(X)\theta_2(U)
-\frac{p-q}{2(n-2)}d\theta(X,U)
+\frac1{n-2}\Ric(X,U).
\]
Since $\Theta^a=\tfrac12(d\theta)^\sharp$, adding and subtracting $\tfrac12d\theta(X,U)$ proves the two general formulas. The block-diagonal Ricci specialization is immediate.
\end{proof}

\begin{corollary}[Partial Ricci form of the obstruction]\label{cor:partial-Ricci}
Assume in addition that $\Ric(T_1,T_2)=0$. Let
$X\in T_1$, $U\in T_2$, and let $e_1,\ldots,e_p$ be a local orthonormal
frame of $T_1$. Then
\begin{equation}\label{eq:partial-Ricci-obstruction}
\boxed{
d\theta(X,U)
=
\frac{n-2}{(p-1)(q-1)}
\sum_{a=1}^{p}R(X,e_a,e_a,U).
}
\end{equation}
Under this block-diagonal Ricci condition the same sum may be written with $W$ in place of $R$.
\end{corollary}

\begin{proof}
By \cref{app:partial-traces},
\[
\sum_{a=1}^{p}R(X,e_a,e_a,U)
=
\frac{(p-1)(q-1)}{n-2}\,d\theta(X,U)
\]
whenever $\Ric(X,U)=0$, proving \eqref{eq:partial-Ricci-obstruction}.
For the Weyl tensor, the Kulkarni--Nomizu correction in
\eqref{eq:weyl-decomp-general} has zero $T_1$-partial trace on
$(X,U)$ when $\Ric(T_1,T_2)=0$, so the two partial traces agree.
\end{proof}

\begin{remark}[Equal ranks]
If $p=q$, then \eqref{eq:symmetric-companion} becomes
\[
S\Theta^s-\Theta^sS+\theta^\sharp\wedge S\theta^\sharp
=\frac1{n-2}[S,\Ric^\sharp].
\]
Hence the Faraday term drops out of the symmetric relation. Under block-diagonal Ricci curvature this further reduces to
\[
S\Theta^s-\Theta^sS=-\theta^\sharp\wedge S\theta^\sharp,
\]
so the mixed symmetric part is then fixed algebraically by the Lee form itself.
\end{remark}

\begin{remark}[Independence of the ordering of the factors]
Interchanging $T_1$ and $T_2$ sends
\[
S\longmapsto -S,
\qquad
\KW\longmapsto-\KW,
\]
while $p$ and $q$ are exchanged. Hence both the coefficient in
\eqref{eq:main-obstruction} and the commutator $[\KW,S]$ are unchanged.
Thus the curvature formula for the Faraday form is intrinsic to the unordered splitting. In the symmetric companion, $s=p-q$ and $[S,\Ric^\sharp]$ both change sign, as does the left-hand side, so \eqref{eq:symmetric-companion} is equally consistent with the swap.
\end{remark}

\begin{remark}[Why the condition $p,q\ge2$ is essential]
If $p=1$ or $q=1$, then $4(p-1)(q-1)=0$ and the two-by-two system is singular. In particular, the Faraday form is no longer algebraically determined by the Weyl commutator through this system. The cases in which one distribution has rank one therefore require a separate analysis and admit warped-type alternatives not governed by the invertible system used here; compare the rank-one discussions in \cite{MoroianuPilca2026,Jiang2026}.
\end{remark}

\begin{corollary}[Characterization of closedness]
For every conformal product with $p,q\ge2$,
\begin{equation}
D\text{ is closed}
\iff
d\theta=0
\iff
[\KW,S]=0.
\end{equation}
Equivalently, $D$ is closed exactly when $\KW$ preserves $T_1$ and $T_2$.
\end{corollary}

\begin{proof}
The first equivalence is the definition of a closed Weyl connection. The second follows from \eqref{eq:main-obstruction}. Since $S$ has the two eigenspaces $T_1,T_2$, an endomorphism commutes with $S$ if and only if it preserves both eigenspaces.
\end{proof}

\begin{corollary}[Pointwise norm identity]\label{cor:energy}
For every conformal product with $p,q\ge2$,
\begin{equation}\label{eq:energy}
\boxed{
|d\theta|^2
=
\frac{(n-2)^2}{32(p-1)^2(q-1)^2}
|[\KW,S]|_{\HS}^2.
}
\end{equation}
\end{corollary}

\begin{proof}
For any two-form $\omega$, our convention \eqref{eq:sharp-convention} gives
\[
|\omega^\sharp|_{\HS}^2=2|\omega|^2.
\]
Take Hilbert--Schmidt norms in \eqref{eq:main-obstruction}.
\end{proof}

\begin{corollary}[Mixed partial Weyl trace]\label{cor:mixed-trace}
Let $(M^n,g)$ carry a conformal product structure with $p=\rk T_1\ge2$ and $q=\rk T_2\ge2$. Let $X\in T_1$ and $U\in T_2$, and let $e_1,\ldots,e_p$ be a local orthonormal frame of $T_1$. Then
\begin{equation}\label{eq:mixed-partial-Weyl}
\boxed{
d\theta(X,U)
=
\frac{n-2}{(p-1)(q-1)}
\sum_{a=1}^p W(X,e_a,e_a,U).
}
\end{equation}
\end{corollary}

\begin{proof}
For $X\in T_1$,
\[
[\KW,S]X
=\KW X-S\KW X
=2(\KW X)_{T_2}.
\]
Thus \cref{thm:main-obstruction} gives
\[
d\theta(X,U)
=
\frac{n-2}{2(p-1)(q-1)}g(\KW X,U).
\]
For an adapted orthonormal frame $e_1,\dots,e_p,f_1,\dots,f_q$,
\[
g(\KW X,U)
=
\sum_{a=1}^pW(X,e_a,e_a,U)
-
\sum_{\alpha=1}^qW(X,f_\alpha,f_\alpha,U).
\]
The Weyl tensor is trace-free, so the two sums add to zero. Hence the displayed expression equals twice the first sum, proving \eqref{eq:mixed-partial-Weyl}.
\end{proof}

\begin{example}[A non-closed example with block-diagonal Ricci tensor]\label{ex:nonclosed-block-Ricci}
The hypothesis $\Ric(T_1,T_2)=0$ does not force the right-hand side of
\eqref{eq:main-obstruction} to vanish.  Let
\[
M=\R^p\times\R^q,\qquad p,q\ge2,
\]
with Euclidean factor metrics, put $t=x_1+y_1$, and choose constants
$c\ne0$ and $\mu\ne0$.  Set
\[
\gamma=\frac{n-2}{q-1}\,c,\qquad
f_1=c\,t,\qquad
f_2=\mu\,e^{\gamma t},
\]
and consider the local conformal-product metric
\[
g=e^{2f_1}g_{\R^p}+e^{2f_2}g_{\R^q}.
\]
The mixed Ricci formula of \cref{app:mixed-Ricci} gives
\begin{align*}
\Ric(X,U)
&=(1-p)X(Uf_1)+(1-q)X(Uf_2)
 +(n-2)X(f_2)U(f_1).
\end{align*}
Both $f_1$ and $f_2$ depend only on $t$, so the only pair that can
contribute is $X=\partial_{x_1}$, $U=\partial_{y_1}$; for all other mixed
pairs every term above vanishes.  For this pair $X(Uf_1)=0$ and
\[
X(Uf_2)=\mu\gamma^2e^{\gamma t},
\qquad
X(f_2)U(f_1)=\mu\gamma c\,e^{\gamma t},
\]
so
\[
\Ric(X,U)=\mu\gamma e^{\gamma t}\bigl((1-q)\gamma+(n-2)c\bigr)=0
\]
by the definition of $\gamma$.  Thus $\Ric(T_1,T_2)=0$, that is,
$[\Rish,S]=0$.

By \eqref{eq:app-Lee}, the adapted Weyl connection has
\[
\theta_1=-\mu\gamma e^{\gamma t}\,dx_1,
\qquad
\theta_2=-c\,dy_1,
\]
so $d\theta_2=0$ and
\[
d\theta
=
d\theta_1
=
\mu\gamma^2e^{\gamma t}\,dx_1\wedge dy_1\ne0.
\]
Consequently \cref{thm:main-obstruction} gives $[\KW,S]\ne0$, and one checks
directly that
\[
\frac{n-2}{(p-1)(q-1)}\sum_{a=1}^{p}R(X,e_a,e_a,U)
=
\mu\gamma^2e^{\gamma t}
=
d\theta(X,U),
\]
in agreement with \cref{cor:partial-Ricci}.  This example shows that even the additional block-diagonal Ricci condition used in Section~3 does not force closedness, and that the Weyl commutator in the universal identity can be nonzero. We make no claim here that this model is Einstein, nor do we claim the existence of a non-closed Einstein conformal product.
\end{example}

\begin{example}[A closed but non-special conformally flat model]\label{ex:hyperbolic}
Let $p,q\ge2$ and
\[
M=\{(x,y)\in\R^p\times\R^q:x_1+y_1>0\},
\qquad t=x_1+y_1.
\]
Put
\begin{equation}
g=t^{-2}(g_{\R^p}+g_{\R^q}).
\end{equation}
After the orthogonal change $u=(x_1+y_1)/\sqrt2$, this is
\[
g=\frac12\,g_{\Hy^n},
\]
so its sectional curvature is $-2$ and
\[
\Ric_g=-2(n-1)g.
\]
In particular $g$ is Einstein and $W=0$. Take as $D$ the Levi-Civita connection of the Euclidean product metric. With the convention \eqref{eq:weyl-connection}, its Lee form relative to $g$ is
\[
\theta=\frac{dt}{t}.
\]
Thus
\[
d\theta=0,
\]
in agreement with \cref{thm:main-obstruction}. However, writing
\[
\theta_1=\frac{dx_1}{t},
\qquad
\theta_2=\frac{dy_1}{t},
\]
one finds
\[
d\theta_1=\frac{dx_1\wedge dy_1}{t^2}\ne0,
\qquad
 d\theta_2=-\frac{dx_1\wedge dy_1}{t^2}\ne0.
\]
Thus closedness of $D$ does not imply specialness, even in an Einstein and conformally flat local model.
\end{example}

\section{Components of the Lee form and specialness}

Throughout this section we assume that the Ricci tensor is block diagonal with respect to the splitting,
\[
\Ric(T_1,T_2)=0,
\qquad
p,q\ge2,
\]
This condition is automatic when $g$ is Einstein. We write
\[
\theta=\theta_1+\theta_2,
\qquad
\theta_i\in T_i^*.
\]
Following \cite{MoroianuPilca2026}, the metric $g$ is called \emph{special} with respect to the conformal product if at least one of $\theta_1,\theta_2$ is closed. Thus specialness is weaker than simultaneous closedness of both components of the Lee form, and it is not equivalent to closedness of the full Weyl connection.

\begin{proposition}[Mixed identity under block-diagonal Ricci curvature]\label{prop:mixed-Lee}
The components of the Lee form satisfy
\begin{equation}\label{eq:mixed-Lee}
\boxed{
(q-1)d\theta_1-(p-1)d\theta_2
=
(n-2)\theta_1\wedge\theta_2.
}
\end{equation}
\end{proposition}

\begin{proof}
The statement is local. On a product chart
$\mathcal U=\mathcal U_1\times\mathcal U_2$, write
\begin{equation}
g=e^{2f_1}g_1+e^{2f_2}g_2,
\end{equation}
with $g_i$ a metric on $\mathcal U_i$ and $f_i\in C^\infty(\mathcal U)$.
By \eqref{eq:app-Lee}, the adapted Weyl connection has
\begin{equation}\label{eq:local-Lee}
\theta_1=-d_1f_2,
\qquad
\theta_2=-d_2f_1.
\end{equation}
For $X\in T_1$ and $U\in T_2$,
\[
d\theta_1(X,U)=X(Uf_2),
\qquad
d\theta_2(X,U)=-X(Uf_1),
\]
and
\[
(\theta_1\wedge\theta_2)(X,U)=X(f_2)U(f_1).
\]
The self-contained mixed Ricci calculation in \cref{app:mixed-Ricci} gives
\begin{equation}\label{eq:mixed-Ricci-local}
\Ric_g(X,U)
=
(1-p)X(Uf_1)+(1-q)X(Uf_2)
+(n-2)X(f_2)U(f_1).
\end{equation}
The left-hand side vanishes by hypothesis, and substitution gives
\eqref{eq:mixed-Lee} on mixed pairs.

Finally, $d_1\theta_1=d_2\theta_2=0$, so
$d\theta_1$, $d\theta_2$, and $\theta_1\wedge\theta_2$ are all of pure
mixed bidegree. Hence equality on mixed pairs proves the two-form identity.
\end{proof}

\begin{proposition}[Decomposition of the Lee-form component derivatives]\label{prop:Lee-decomp}
Let $F=d\theta$. Then
\begin{equation}\label{eq:dtheta1-decomp}
\boxed{
d\theta_1
=
\frac{p-1}{n-2}F+\theta_1\wedge\theta_2,
}
\end{equation}
and
\begin{equation}\label{eq:dtheta2-decomp}
\boxed{
d\theta_2
=
\frac{q-1}{n-2}F-\theta_1\wedge\theta_2.
}
\end{equation}
\end{proposition}

\begin{proof}
Set $\Phi_1=d\theta_1$, $\Phi_2=d\theta_2$, and $\Xi=\theta_1\wedge\theta_2$.
Since $F=\Phi_1+\Phi_2$, \cref{prop:mixed-Lee} becomes
\[
(q-1)\Phi_1-(p-1)\Phi_2=(n-2)\Xi.
\]
Substituting $\Phi_2=F-\Phi_1$ gives
\[
(n-2)\Phi_1=(p-1)F+(n-2)\Xi,
\]
which is \eqref{eq:dtheta1-decomp}; then \eqref{eq:dtheta2-decomp} follows
from $\Phi_2=F-\Phi_1$.
\end{proof}

\begin{example}[Closed does not imply special]
For the metric in \cref{ex:hyperbolic}, one has $F=d\theta=0$ but
\[
\theta_1\wedge\theta_2
=\frac{dx_1\wedge dy_1}{t^2}\ne0.
\]
Equations \eqref{eq:dtheta1-decomp}--\eqref{eq:dtheta2-decomp} reduce to
\[
d\theta_1=\theta_1\wedge\theta_2,
\qquad
 d\theta_2=-\theta_1\wedge\theta_2,
\]
exactly as obtained by direct differentiation.
\end{example}

\begin{corollary}[Weighted norm identity]
Pointwise on $M$,
\begin{equation}\label{eq:weighted-norm}
\boxed{
\begin{aligned}
&(q-1)|d\theta_1|^2+(p-1)|d\theta_2|^2\\
&\qquad=
\frac{(p-1)(q-1)}{n-2}|d\theta|^2
+(n-2)|\theta_1\wedge\theta_2|^2.
\end{aligned}}
\end{equation}
Consequently, using \cref{cor:energy},
\begin{equation}\label{eq:weighted-Weyl}
\boxed{
\begin{aligned}
&(q-1)|d\theta_1|^2+(p-1)|d\theta_2|^2\\
&\qquad=
\frac{n-2}{32(p-1)(q-1)}|[\KW,S]|_{\HS}^2
+(n-2)|\theta_1\wedge\theta_2|^2.
\end{aligned}}
\end{equation}
\end{corollary}

\begin{proof}
Put
\[
a=\frac{p-1}{n-2},
\qquad
b=\frac{q-1}{n-2},
\qquad a+b=1,
\]
and $\Xi=\theta_1\wedge\theta_2$. By \cref{prop:Lee-decomp},
\[
d\theta_1=aF+\Xi,
\qquad
 d\theta_2=bF-\Xi.
\]
Expanding the weighted sum, the cross term has coefficient
\[
2\bigl((q-1)a-(p-1)b\bigr)=0.
\]
Moreover,
\[
(q-1)a^2+(p-1)b^2
=
\frac{(p-1)(q-1)}{n-2},
\]
while $(q-1)+(p-1)=n-2$. This proves \eqref{eq:weighted-norm}. Formula \eqref{eq:weighted-Weyl} follows from \eqref{eq:energy}.
\end{proof}

\begin{proposition}[A differential consequence]\label{prop:closed-wedge}
The two-form $\theta_1\wedge\theta_2$ is closed. Moreover,
\begin{equation}
\theta_1\wedge d\theta=0,
\qquad
 d\theta\wedge\theta_2=0.
\end{equation}
\end{proposition}

\begin{proof}
Taking $d$ in \eqref{eq:mixed-Lee} gives
\[
0=(n-2)d(\theta_1\wedge\theta_2),
\]
so $d(\theta_1\wedge\theta_2)=0$. Using \cref{prop:Lee-decomp},
\begin{align*}
0
&=d(\theta_1\wedge\theta_2)\\
&=d\theta_1\wedge\theta_2-\theta_1\wedge d\theta_2\\
&=
\frac{p-1}{n-2}d\theta\wedge\theta_2
-
\frac{q-1}{n-2}\theta_1\wedge d\theta.
\end{align*}
The first term has bidegree $(1,2)$ and the second bidegree $(2,1)$ with
respect to $T_1^*\oplus T_2^*$, so they lie in distinct summands of
$\Lambda^3T^*M$; hence both vanish.
\end{proof}

\begin{proposition}[Reduction where both components of the Lee form are nonzero]
Let
\[
\Omega=\{x\in M:\theta_1(x)\ne0,\ \theta_2(x)\ne0\}.
\]
There exists a smooth function $\varphi\in C^\infty(\Omega)$ such that
\begin{equation}\label{eq:F-phi}
\boxed{
d\theta=\varphi\,\theta_1\wedge\theta_2.
}
\end{equation}
Furthermore,
\begin{equation}\label{eq:dphi-span}
d\varphi\in\operatorname{span}\{\theta_1,\theta_2\}
\qquad\text{on }\Omega.
\end{equation}
\end{proposition}

\begin{proof}
At each point of $\Omega$, the nonzero one-forms $\theta_1\in T_1^*$ and $\theta_2\in T_2^*$ lie in complementary summands, hence $\theta_1\wedge\theta_2\ne0$. Also $F=d\theta$ belongs to $T_1^*\wedge T_2^*$. The identity $\theta_1\wedge F=0$ from \cref{prop:closed-wedge} forces the $T_1^*$ component of $F$ to be proportional to $\theta_1$. Hence $F=\theta_1\wedge\eta$ for some $\eta\in T_2^*$. The identity $F\wedge\theta_2=0$ then forces $\eta$ to be proportional to $\theta_2$, yielding \eqref{eq:F-phi}. Moreover, on $\Omega$ one may write
\[
\varphi=\frac{\langle d\theta,\theta_1\wedge\theta_2\rangle}{|\theta_1\wedge\theta_2|^2},
\]
so $\varphi$ is smooth.

Since $F$ and $\theta_1\wedge\theta_2$ are closed,
\[
0=dF=d\varphi\wedge\theta_1\wedge\theta_2.
\]
Because $\theta_1$ and $\theta_2$ are linearly independent on $\Omega$, this is equivalent to \eqref{eq:dphi-span}.
\end{proof}

\begin{corollary}[Characterization of specialness on $\Omega$]
On $\Omega$,
\begin{align}
d\theta_1
&=
\left(1+\frac{p-1}{n-2}\varphi\right)
\theta_1\wedge\theta_2,
\\
d\theta_2
&=
\left(-1+\frac{q-1}{n-2}\varphi\right)
\theta_1\wedge\theta_2.
\end{align}
Thus
\begin{equation}\label{eq:special-thresholds}
 d\theta_1=0
\iff
\varphi=-\frac{n-2}{p-1},
\qquad
 d\theta_2=0
\iff
\varphi=\frac{n-2}{q-1}.
\end{equation}
\end{corollary}

\begin{proof}
Substitute \eqref{eq:F-phi} in \eqref{eq:dtheta1-decomp}--\eqref{eq:dtheta2-decomp}. Since $\theta_1\wedge\theta_2\ne0$ on $\Omega$, the threshold characterization follows.
\end{proof}

\begin{corollary}[The components of the Lee form cannot both be closed on $\Omega$]
On $\Omega$, the two components of the Lee form cannot both be closed.
\end{corollary}

\begin{proof}
The two threshold values in \eqref{eq:special-thresholds} have opposite signs:
\[
-\frac{n-2}{p-1}<0<
\frac{n-2}{q-1}.
\]
Hence a single value of $\varphi$ cannot satisfy both conditions.
\end{proof}

\begin{corollary}[Rank restriction on the off-diagonal Weyl contraction]\label{cor:rank-one}
On $\Omega$, let $B_W:T_1\to T_2$ be the off-diagonal block of $\KW$. Then
\begin{equation}\label{eq:BW-rank-one}
B_W(X)
=
\frac{2(p-1)(q-1)}{n-2}
\varphi\,\theta_1(X)\theta_2^\sharp.
\end{equation}
In particular,
\[
\rk B_W\le1.
\]
\end{corollary}

\begin{proof}
For $X\in T_1$, $[\KW,S]X=2B_WX$. Hence \cref{thm:main-obstruction} gives
\[
(d\theta)^\sharp X
=
\frac{n-2}{2(p-1)(q-1)}B_WX.
\]
On the other hand, \eqref{eq:F-phi} yields
\[
(d\theta)^\sharp X
=
\varphi\,\theta_1(X)\theta_2^\sharp.
\]
Comparison proves \eqref{eq:BW-rank-one}.
\end{proof}

\begin{remark}
The results of this section do not settle whether the Einstein condition forces a conformal product to be special. They nevertheless separate two mechanisms: the Faraday term $d\theta$, controlled by the Weyl commutator through \cref{thm:main-obstruction}, and the interaction term $\theta_1\wedge\theta_2$. The local hyperbolic model shows that the first may vanish while the second does not.
\end{remark}

\section{Four-dimensional closedness and compact rigidity}

We now assume
\[
n=4,
\qquad
p=q=2.
\]
The first result is local and does not require compactness.

\begin{proposition}[The induced ambi-Hermitian structure]\label{prop:ambihermitian}
Locally, the rank-$(2,2)$ conformal product determines two orthogonal integrable complex structures $J_+$ and $J_-$ inducing opposite orientations and satisfying
\[
DJ_+=DJ_-=0,
\qquad
S=-J_+J_-.
\]
\end{proposition}

\begin{proof}
Work on a local orientation cover, if necessary, and orient both plane bundles
$T_1$ and $T_2$. Let $I_i$ be rotation by $\pi/2$ on the oriented conformal
plane bundle $T_i$. Since $D$ preserves the splitting and the conformal
structures of the two plane bundles, it preserves $I_1$ and $I_2$. Set
\[
J_+=I_1\oplus I_2,
\qquad
J_-=I_1\oplus(-I_2).
\]
Then $J_\pm^2=-\Id$, both are orthogonal, and $DJ_\pm=0$. Moreover, on $T_1$ the product $J_+J_-$ equals $-\Id$, while on $T_2$ it equals $\Id$; hence $S=-J_+J_-$. This is the eigenspace decomposition associated with a commuting ambi-Hermitian pair; compare \cite[Definition~1 and Lemma~1]{ApostolovCalderbankGauduchon2016}. A torsion-free
connection preserving an almost-complex structure forces its Nijenhuis tensor
to vanish, so $J_\pm$ are integrable. Reversing the complex orientation on
exactly one of the two oriented planes reverses the induced four-dimensional
orientation; hence $J_+$ and $J_-$ have opposite orientations.

If $\omega_\pm=g(J_\pm\cdot,\cdot)$, the standard Weyl calculation gives
\[
d\omega_\pm=-2\theta\wedge\omega_\pm;
\]
see \cite[Lemma~5.7]{BelgunMoroianu2011}. Thus the same adapted Weyl
connection is compatible with both Hermitian structures. In the convention of \cite{ApostolovCalderbankGauduchon2016}, where $d\omega_\pm=-2\theta^g_\pm\wedge\omega_\pm$, this says precisely that
\[
\theta^g_+=\theta^g_-=\theta.
\]
Accordingly, the rank-$(2,2)$ conformal-product structure gives, locally (or on the orientation cover), an ambi-Hermitian structure whose two canonical Weyl connections coincide.
\end{proof}

\begin{lemma}[Ambi-Hermitian Ricci invariance]\label{lem:ambi-Ricci}
Let $\mathcal A=\Rish$ for a rank-$(2,2)$ conformal product. The following are equivalent:
\begin{enumerate}
\item $\Ric$ is invariant under both $J_+$ and $J_-$;
\item $\mathcal A$ commutes with both $J_+$ and $J_-$;
\item there exist smooth functions $\mu_1,\mu_2$ such that
\begin{equation}\label{eq:two-Ricci-eigenvalues}
\mathcal A=\mu_1\Id_{T_1}\oplus\mu_2\Id_{T_2},
\qquad\text{equivalently}\qquad
\Ric=\mu_1 g|_{T_1}\oplus\mu_2 g|_{T_2}.
\end{equation}
\end{enumerate}
In particular, any of these conditions implies $\Ric(T_1,T_2)=0$.
\end{lemma}

\begin{proof}
For an orthogonal complex structure $J$, invariance of the symmetric tensor $\Ric$ is equivalent to $\mathcal A J=J\mathcal A$. If $\mathcal A$ commutes with $J_+$ and $J_-$, then it commutes with $S=-J_+J_-$ and therefore preserves $T_1$ and $T_2$. Its restriction to each oriented two-plane commutes with the rotation $I_i$ by $\pi/2$. A self-adjoint endomorphism of a Euclidean two-plane commuting with that rotation is a scalar multiple of the identity. This gives \eqref{eq:two-Ricci-eigenvalues}. The converse is immediate because $J_\pm$ preserve each summand and act there by $\pm I_i$. The final assertion follows from the block form.
\end{proof}

\begin{example}[Product of oriented surfaces]
For a product of oriented Riemannian surfaces $(\Sigma_1,g_1)\times(\Sigma_2,g_2)$, the two structures $J_\pm$ in \cref{prop:ambihermitian} are the familiar product complex structures $I_1\oplus I_2$ and $I_1\oplus(-I_2)$. If the product splitting is Levi-Civita parallel, both are K\"ahler.
\end{example}

\begin{theorem}[Closedness under ambi-Hermitian Ricci invariance]\label{thm:4d-closed}
Let $(M^4,g)$ carry a rank-$(2,2)$ conformal product structure with adapted Weyl connection $D$. Assume that the Ricci tensor is invariant under both induced Hermitian structures $J_+$ and $J_-$. By \cref{lem:ambi-Ricci}, this is equivalent to
\begin{equation}\label{eq:4d-Ricci-block-scalar}
\Ric=\mu_1 g|_{T_1}\oplus\mu_2 g|_{T_2}
\end{equation}
for smooth functions $\mu_1,\mu_2$. Then
\begin{equation}\label{eq:4d-closed}
\boxed{d\theta=0.}
\end{equation}
\end{theorem}

\begin{proof}
Use the Lee-form convention of Apostolov--Calderbank--Gauduchon,
\[
d\omega_\pm=-2\theta^g_\pm\wedge\omega_\pm.
\]
By \cref{prop:ambihermitian}, the two canonical Hermitian Weyl connections agree with the adapted Weyl connection $D$, and therefore
\[
\theta^g_+=\theta^g_-=\theta.
\]
For an ambi-Hermitian metric whose Ricci tensor is invariant under both complex structures, Apostolov--Calderbank--Gauduchon show that the two half-Weyl tensors are degenerate and, with the orientation induced by $J_+$ fixed, $d\theta^g_+$ is anti-self-dual while $d\theta^g_-$ is self-dual \cite[Proposition~1 and Section~3.3]{ApostolovCalderbankGauduchon2016}. Hence the same two-form $d\theta$ is both anti-self-dual and self-dual. It follows that $d\theta=0$.
\end{proof}

\begin{remark}[Curvature interpretation and consistency with Section~2]
The preceding proof is the classical ambi-Hermitian argument. We record how the same half-Weyl information appears in the mixed Weyl contraction of Section~2.

For a Hermitian surface, Apostolov--Calderbank--Gauduchon write the self-dual Weyl operator in the form
\[
\mathcal W^+
=\frac18\kappa^g(\omega\otimes\omega)_0
   +J(d\theta^g)_+\odot\omega
\]
(up to the standard identification of symmetric endomorphisms of $\Lambda^+$); see \cite[Section~1.2]{ApostolovCalderbankGauduchon2016}. Normalize the fundamental form by $\widehat\omega=\omega/\sqrt2$ and complete it to an orthonormal basis $(\widehat\omega,\widehat\gamma,\widehat\delta)$ of $\Lambda^+$. Then
\[
(\widehat\omega\otimes\widehat\omega)_0
=\operatorname{diag}\!\left(\frac23,-\frac13,-\frac13\right).
\]
Choosing $\widehat\gamma$ in the direction of the $\omega$-orthogonal part of $J(d\theta^g)_+$ when it is nonzero, and absorbing the fixed normalization constants into $a$ and $c$, where $a$ comes from the $\frac18\kappa^g$ term and $c$ from the mixed term $J(d\theta^g)_+\odot\omega$, the matrix of $\mathcal W^+$ is
\[
\begin{pmatrix}
 2a/3 & c & 0\\
 c & -a/3 & 0\\
 0 & 0 & -a/3
\end{pmatrix}.
\]
Its eigenvalues are $-a/3$ and
\[
\frac{a/3\pm\sqrt{a^2+4c^2}}{2}.
\]
If $\mathcal W^+$ is degenerate, then either the two eigenvalues in the displayed pair coincide, which forces $a=c=0$ and hence $\mathcal W^+=0$, or one of them equals $-a/3$, which forces $c=0$. In either case $c=0$. Thus $(d\theta^g)_+=0$, and the fundamental form spans an eigenline of $\mathcal W^+$. The same statement with the opposite orientation applies to $J_-$. Consequently, for an adapted orthonormal frame $e_1,e_2\in T_1$ and $e_3,e_4\in T_2$,
\[
\omega_+=e^{12}+e^{34}\in\Lambda^+,
\qquad
\omega_-=e^{12}-e^{34}\in\Lambda^-,
\]
and
\[
\mathcal W^+\omega_+\in\mathbb R\,\omega_+,
\qquad
\mathcal W^-\omega_-\in\mathbb R\,\omega_-.
\]
Let $\mathcal W$ denote the Weyl curvature operator on $\Lambda^2T^*M$. It preserves the orthogonal splitting $\Lambda^2=\Lambda^+\oplus\Lambda^-$, so its restrictions are precisely $\mathcal W^\pm$. With any fixed curvature-operator convention, $\langle\mathcal W\eta,\zeta\rangle$ agrees with the corresponding component of $W$ up to one global sign. Since only vanishing is used below, that sign plays no role.

Put
\[
\eta_2^+=e^{13}-e^{24},\quad
\eta_2^-=e^{13}+e^{24},\quad
\eta_3^+=e^{14}+e^{23},\quad
\eta_3^-=e^{14}-e^{23}.
\]
Using $e^{12}=\tfrac12(\omega_++\omega_-)$ and $e^{23}=\tfrac12(\eta_3^+-\eta_3^-)$, the orthogonality of $\Lambda^+$ and $\Lambda^-$ gives
\[
W_{1223}
=\frac14\left(
\langle\mathcal W^+\omega_+,\eta_3^+\rangle
-\langle\mathcal W^-\omega_-,\eta_3^-\rangle
\right)=0.
\]
The $\eta_2^\pm$ decomposition similarly gives $W_{1224}=0$, and replacing $e^{12}$ by $e^{21}=-e^{12}$ gives
\[
W_{2113}=W_{2114}=0.
\]
For $p=q=2$, these are precisely the four mixed components occurring in \cref{cor:mixed-trace}. Thus the classical half-Weyl degeneracy is consistent with the curvature formula of Section~2 and yields $[\mathcal K_W,S]=0$.
\end{remark}

\begin{remark}[Relation with the local ambi-Hermitian classification]
Apostolov--Calderbank--Gauduchon also classify, on a dense open set, Einstein four-manifolds for which both half-Weyl tensors are degenerate \cite[Corollary~1]{ApostolovCalderbankGauduchon2016}. That local classification does not imply the compact rigidity theorem below. In particular, it concerns the local form of the Einstein metric $g$, whereas our conclusion is that the given $D$-parallel splitting is actually $\nabla^g$-parallel. The distinction is already visible in the constant-curvature branch: \cref{ex:hyperbolic} has $d\theta=0$ with a nonzero Lee form and is not special. The global argument below is therefore needed to pass from closedness of $D$ to $D=\nabla^g$.
\end{remark}

\begin{corollary}[Einstein case]\label{cor:4d-einstein-closed}
Every rank-$(2,2)$ Einstein conformal product has closed adapted Weyl connection.
\end{corollary}

\begin{proof}
For an Einstein metric, \eqref{eq:4d-Ricci-block-scalar} holds with $\mu_1=\mu_2=\lambda$.
\end{proof}

\begin{remark}[Why a curvature hypothesis remains necessary]
The Riemannian Goldberg--Sachs input above requires $J$-invariance of the Ricci tensor; a Hermitian structure alone does not imply the Weyl-eigenform conclusion. Belgun and Moroianu construct non-closed conformal products in dimension four \cite{BelgunMoroianu2011}, so rank $(2,2)$ by itself cannot imply \eqref{eq:4d-closed}. The local theorem identifies a natural intermediate curvature condition: ambi-Hermitian Ricci invariance supplies the half-Weyl degeneracy needed for the classical closedness argument and the curvature interpretation identifies the same vanishing in the mixed Weyl contraction. The structural identity \eqref{eq:main-obstruction} itself is assumption-free. The Ricci-invariance assumption is exactly what is needed in the ambi-Hermitian curvature argument; we do not attempt to classify the non-Einstein case.
\end{remark}

\begin{corollary}\label{cor:4d-KW}
Under the assumptions of \cref{thm:4d-closed},
\[
[\KW,S]=0.
\]
Moreover,
\begin{equation}\label{eq:4d-Lee-components}
d\theta_1=\theta_1\wedge\theta_2,
\qquad
 d\theta_2=-\theta_1\wedge\theta_2.
\end{equation}
Consequently,
\[
d\theta_1=0
\iff
d\theta_2=0
\iff
\theta_1\wedge\theta_2=0.
\]
Moreover, for $X\in T_1$ and $U\in T_2$,
\[
(\nabla_X\theta)(U)=(\nabla_U\theta)(X)=\theta_1(X)\theta_2(U).
\]
\end{corollary}

\begin{proof}
For $n=4$ and $p=q=2$, \cref{thm:main-obstruction} becomes
\[
(d\theta)^\sharp=\frac12[\KW,S].
\]
Hence \cref{thm:4d-closed} gives the first assertion. Equations \eqref{eq:4d-Lee-components} follow from \cref{prop:Lee-decomp} with $d\theta=0$. The mixed covariant-derivative formula follows from \cref{cor:mixed-nabla-theta} with $p=q=2$ and $d\theta=0$.
\end{proof}

\begin{example}[Local closed but non-special four-dimensional model]
Take $p=q=2$ in \cref{ex:hyperbolic}. Then $d\theta=0$ but $\theta_1\wedge\theta_2\ne0$, so neither Lee component is closed. Hence \cref{thm:4d-closed} asserts closedness only; it does not imply specialness.
\end{example}

We now assume that $M$ is compact.

\begin{lemma}[Exactness]\label{lem:4d-exact}
Let $(M^4,g)$ be compact and Einstein and let $D$ be a rank-$(2,2)$ conformal product structure. Then $D$ is exact. Hence there is $\phi\in C^\infty(M)$ such that
\begin{equation}\label{eq:h-conformal}
D=\nabla^h,
\qquad
h=e^{2\phi}g,
\end{equation}
and $h$ has reducible holonomy.
\end{lemma}

\begin{proof}
By \cref{cor:4d-einstein-closed}, $D$ is closed. Suppose, for contradiction, that it is non-exact. Since $D$ preserves the rank-$(2,2)$ splitting, its holonomy is reducible.

Assume first that $D$ is non-flat. We use the terminology of Belgun--Flamencourt--Moroianu \cite[Definition~4.1]{BelgunFlamencourtMoroianu2025}: on a compact conformal manifold an LCP structure is a Weyl connection that is closed, non-exact, non-flat, and has reducible holonomy. Thus the present branch is precisely an LCP structure compatible with the Einstein metric $g$. Their Theorem~4.5 states that no LCP structure is compatible with a compact Einstein metric. Hence this branch is impossible.

It remains to exclude the flat non-exact branch, which is excluded from the definition of an LCP structure. Because $D$ is closed, its lift to the simply connected universal cover is the Levi-Civita connection of a $D$-parallel flat metric $h_0$ in the lifted conformal class, and the deck group acts on $(\widetilde M,h_0)$ by similarities. Exactness is equivalent to all deck similarities being isometries. Since $M$ is compact and $D$ is assumed non-exact, Fried's classification of closed similarity manifolds \cite{Fried1980} applies. In the form recalled in \cite[Section~4]{BelgunFlamencourtMoroianu2025}, the non-exact flat case is radiant and its developing image is the punctured Euclidean model; in particular, the universal cover is diffeomorphic to
\[
\R^4\setminus\{0\}.
\] On the other hand, flatness of $D$ means locally that $D$ is the Levi-Civita connection of a flat metric conformal to $g$. Weyl curvature is conformally invariant, so $W_g=0$. Since $g$ is Einstein, it has constant sectional curvature. The lift of $g$ to the universal cover is complete and simply connected, hence its underlying manifold is diffeomorphic to $S^4$ in the positive-curvature case and to $\R^4$ in the zero or negative-curvature cases. Neither is diffeomorphic to $\R^4\setminus\{0\}$; for instance, the latter has nontrivial third homology. This contradiction excludes the flat branch as well.

Therefore $D$ is exact, so $\theta=d\phi$ for some global function $\phi$, and \eqref{eq:h-conformal} follows from \eqref{eq:weyl-connection}. Since $D$ preserves $T_1$ and $T_2$, these distributions are $\nabla^h$-parallel and $h$ has reducible holonomy.
\end{proof}

\begin{lemma}[One-factor dependence and warped reduction]\label{lem:warped-reduction}
Assume that $\phi$ in \eqref{eq:h-conformal} is nonconstant. On the universal cover, after possibly interchanging the factors, the lifted conformal factor depends on only one de Rham factor. Consequently the lifted Einstein metric has the form
\begin{equation}\label{eq:warped-reduction}
\widetilde g=g_B+\psi^2 h_F,
\end{equation}
where
\[
\dim B=\dim F=2,
\qquad
\psi=e^{-\phi}>0,
\qquad
 g_B=\psi^2h_B.
\]
The metric $g_B$ is complete.
\end{lemma}

\begin{proof}
The metric $h$ is complete because it is a smooth Riemannian metric on compact $M$. Thus its lifted metric on the simply connected universal cover is complete and reducible. By the global de Rham theorem,
\[
(\widetilde M,\widetilde h)
=(B^2,h_B)\times(F^2,h_F).
\]
The lift of $\phi$ is bounded because it comes from a function on compact $M$. Write
\[
\widetilde g=f^{-2}(h_B+h_F),\qquad f=e^{\phi}>0.
\]
Both $(B,h_B)$ and $(F,h_F)$ are complete. Suppose that $f$ depends nontrivially on both factors. Since both factors have dimension two, it is important to use the full four-dimensional alternative of K\"uhnel--Rademacher rather than the simplified complete-product corollary. By \cite[Theorem~3.2]{KuhnelRademacher2016},
\[
f(y,x)=a(y)+b(x),
\]
where $a$ and $b$ are both nonconstant. There are two possible subcases.

First assume that $(B,h_B)$ and $(F,h_F)$ have constant Gaussian curvatures $K_B$ and $K_F=-K_B$. The product equations in this branch give
\[
\Hess^{h_B}a=\rho_B h_B,
\qquad
\Hess^{h_F}b=\rho_F h_F
\]
for smooth functions $\rho_B$ and $\rho_F$. Applying the Ricci identity to the first equation and using the surface curvature formula gives
\[
d\rho_B=-K_B\,da.
\]
Since $K_B$ is constant, this integrates globally on the connected factor $B$ to
\[
\rho_B=-K_Ba+c_B,
\qquad
\Hess^{h_B}a=(-K_Ba+c_B)h_B
\]
for a constant $c_B$. Similarly,
\[
\Hess^{h_F}b=(-K_Fb+c_F)h_F
\]
for a constant $c_F$; no relation between $c_B$ and $c_F$ is needed below.

Because $f$ is bounded on $B\times F$, each of $a$ and $b$ is bounded: for example, fixing $x_0\in F$ gives $a(y)=f(y,x_0)-b(x_0)$. If $K_B=K_F=0$, the complete simply connected factors are Euclidean planes and
\[
a(x)=\frac{c_B}{2}|x|^2+\ell(x)+C
\]
for a linear function $\ell$ and a constant $C$. Boundedness forces $c_B=0$ and $\ell=0$, so $a$ is constant; the same argument applies to $b$, contradicting the assumption that both are nonconstant. If $K_B\ne0$, one of the two curvatures is negative. On the corresponding complete simply connected hyperbolic factor, writing $\kappa=-K>0$ and adding a constant to the relevant bounded function reduces its Hessian equation along every complete unit-speed geodesic to
\[
u''=\kappa u.
\]
The only bounded solution on $\mathbb R$ is $u\equiv0$. Since every point of the factor lies on such a complete geodesic, the original function is constant on that factor. This again contradicts the assumption that both $a$ and $b$ are nonconstant.

It remains to consider the exceptional two-surface branch in \cite[Theorem~3.2]{KuhnelRademacher2016}. In this branch both Gaussian curvatures are nonconstant: in the notation used in the derivation there, $K_B$ and $K_F$ are proportional to the nonconstant functions $a$ and $b$ with a nonzero proportionality constant; otherwise one falls back into the constant-curvature branch. They satisfy
\[
\Hess^{h_B}K_B=\frac{\Delta_{h_B}K_B}{2}h_B,
\qquad
\Hess^{h_F}K_F=\frac{\Delta_{h_F}K_F}{2}h_F.
\]
To see the relevant global consequence, define on $M$ the smooth functions
\[
\kappa_i(x):=\sec_h(T_i(x)),\qquad i=1,2,
\]
where $\sec_h(T_i(x))$ denotes the $h$-sectional curvature of the two-plane $T_i(x)\subset T_xM$. Because $h$ and the two plane distributions are globally defined, $\kappa_i$ are smooth functions on the compact manifold $M$. On the universal cover their lifts are exactly $K_B\circ\operatorname{pr}_B$ and $K_F\circ\operatorname{pr}_F$. Since the covering map is surjective, $K_B$ and $K_F$ therefore have the same value sets as $\kappa_1$ and $\kappa_2$, respectively. In particular, each curvature function attains a minimum and a maximum on its complete factor. Since both curvatures are nonconstant, Tashiro's two-stationary-point theorem \cite[Theorem~1]{Tashiro1965} places each factor in the spherical case: both $B$ and $F$ are compact surfaces diffeomorphic to $S^2$, with the two extrema corresponding to the poles. The product $B\times F$ is then compact, while $f=a+b$ is everywhere positive and depends on both factors. This is excluded by the compact product result \cite[Corollary~3.4]{KuhnelRademacher2016}, whose proof explicitly treats the exceptional two-dimensional extremal branch.

Both subcases with genuine dependence on the two factors are impossible. Hence $f$, and therefore $\phi$, depends on only one factor. After interchanging $B$ and $F$ if necessary, we may assume that it depends only on $B$. The bounded-factor idea is closely related to the argument in \cite[Lemma~5.3]{Moroianu2019}; the additional discussion above is needed here because in dimension $2+2$ K\"uhnel--Rademacher's Theorem~3.2 contains the extremal-surface alternative.

Since
\[
\widetilde g=e^{-2\phi}(h_B+h_F),
\]
putting $\psi=e^{-\phi}$ and $g_B=\psi^2h_B$ gives \eqref{eq:warped-reduction}. Finally, $\psi$ is bounded above and below by positive constants, so $g_B$ is complete whenever $h_B$ is complete.
\end{proof}

\begin{lemma}[Compactness of the two-dimensional base]\label{lem:base-compact}
If the positive warping function $\psi$ in \eqref{eq:warped-reduction} is nonconstant, then the complete surface $(B,g_B)$ is diffeomorphic to $S^2$. In particular, it is compact.
\end{lemma}

\begin{proof}
The base component of the Einstein warped-product equation is
\begin{equation}
\Ric_{g_B}-2\psi^{-1}\Hess^{g_B}\psi
=\lambda g_B.
\end{equation}
Since $B$ is a surface, $\Ric_{g_B}=K_Bg_B$, and hence
\begin{equation}
\Hess^{g_B}\psi
=
\frac{\psi}{2}(K_B-\lambda)g_B.
\end{equation}
Thus $\psi$ is a concircular scalar field.

Because the covering map $\widetilde M\to M$ is surjective, the lifted function has the same value set as the original function on compact $M$. Since the lift depends only on $B$, its value set is also exactly the value set of $\psi:B\to\mathbb R$. Hence $\psi$ attains a minimum and a maximum on $B$. If it is nonconstant, these give two distinct stationary points. We apply the two-stationary-point case of Tashiro's classification of complete manifolds carrying a concircular scalar field \cite[Theorem~1]{Tashiro1965}: a complete manifold with such a nonconstant field and two stationary points belongs to the spherical rotational case; in dimension two the underlying surface is diffeomorphic to $S^2$, the stationary points are the two poles, and the regular level sets are circles. Only this topological and normal-form conclusion is needed below.
\end{proof}

\begin{lemma}[Exclusion of the nonconstant exact branch]\label{lem:warping-collapse}
There is no nonconstant $\phi$ in \eqref{eq:h-conformal} satisfying the compact Einstein rank-$(2,2)$ hypotheses.
\end{lemma}

\begin{proof}
Assume that $\phi$ is nonconstant. By \cref{lem:warped-reduction}, the lifted Einstein metric has the warped form \eqref{eq:warped-reduction} with nonconstant $\psi=e^{-\phi}$. By \cref{lem:base-compact}, $(B,g_B)$ is a compact surface in the spherical two-stationary-point case of Tashiro's classification. We use a direct argument; compactness of the total warped product $B\times_\psi F$ is not needed.

Set $r:=\psi$. The Einstein warped-product equation in the base directions is
\[
\Ric_{g_B}-2r^{-1}\Hess^{g_B} r=\lambda g_B.
\]
Since $\Ric_{g_B}=K_Bg_B$, this is
\begin{equation}\label{eq:concircular2}
\Hess^{g_B} r
=\frac r2(K_B-\lambda)g_B.
\end{equation}
On the regular set $dr\ne0$, put $\mathcal V:=|\nabla^{g_B} r|^2$. Differentiating $\mathcal V$ and using \eqref{eq:concircular2} gives
\[
d\mathcal V=2\Hess^{g_B}r(\nabla^{g_B}r,\cdot)
=r(K_B-\lambda)\,dr.
\]
Hence $\mathcal V$ is locally a function of $r$. Because \eqref{eq:concircular2} is pure trace, the preceding identity yields
\[
\Hess^{g_B}r=\frac12\mathcal V'(r)g_B.
\]
Choosing an angular coordinate $\vartheta$ along the regular level sets and fixing its scale once and for all, the metric takes the standard form
\begin{equation}
g_B=\frac{dr^2}{\mathcal V(r)}+\mathcal V(r)d\vartheta^2.
\end{equation}
For this metric
\[
K_B=-\frac12\mathcal V''(r).
\]
Substitution in \eqref{eq:concircular2} gives
\[
r\mathcal V''+2\mathcal V'+2\lambda r=0,
\]
and hence
\begin{equation}\label{eq:Q-solution}
\mathcal V(r)=k-\frac{\lambda}{3}r^2+\frac{C_0}{r}.
\end{equation}
Here the constant term has been denoted by $k$ because the fibre equation of the Einstein warped product shows that $(F,h_F)$ is Einstein and, since $\dim F=2$, has constant Gaussian curvature $k$; explicitly,
\[
k=r\Delta_{g_B} r+|\nabla^{g_B} r|^2+\lambda r^2.
\]
Using $\Delta_{g_B} r=\mathcal V'(r)$ in \eqref{eq:Q-solution} verifies that the right-hand side is indeed the constant term $k$. Equivalently, \eqref{eq:Q-solution} is the $2+2$ Riemannian specialization of the integrated equation in \cite[Corollary~5.4, equation~(29)]{KuhnelRademacher2016}. We keep the short derivation because it fixes the normalization used in the endpoint argument below.

Let
\[
a=\min_B r,\qquad b=\max_B r,
\qquad 0<a<b.
\]
Tashiro's two-stationary-point case identifies these extrema with the two poles of the compact surface, and the regular set is an annulus foliated by the level circles of $r$. Thus
\[
\mathcal V(a)=\mathcal V(b)=0.
\]
Both endpoint zeros are simple. Indeed, if at an endpoint $r_0$ one had $\mathcal V(r)=O((r-r_0)^2)$, then the radial distance
\[
\int_{r_0}^{r_0+\varepsilon}\frac{dr}{\sqrt{\mathcal V(r)}}
\]
would diverge, contradicting the fact that $r_0$ is a pole of the compact smooth surface at finite distance. Equivalently, smooth polar closing requires the radial profile $G=\sqrt{\mathcal V}$ to have nonzero first derivative at the pole.

We now make the endpoint condition explicit. Near either zero $r_0$ of $\mathcal V$, let $\rho$ be the $g_B$-distance from the pole. Since
\[
d\rho=\frac{dr}{\sqrt{\mathcal V(r)}},
\]
the expansion
$\mathcal V(r)=|\mathcal V'(r_0)|\,|r-r_0|+O(|r-r_0|^2)$ gives
\[
\mathcal V(r)
=
\frac{|\mathcal V'(r_0)|^2}{4}\rho^2+O(\rho^4).
\]
Thus smooth polar closing requires the period of $\vartheta$ to be
$4\pi/|\mathcal V'(r_0)|$. The same angular coordinate $\vartheta$ is used at both poles, so the two absolute derivatives agree. Since $\mathcal V>0$ on $(a,b)$, the endpoint derivatives have opposite signs:
\begin{equation}\label{eq:endpoint-smoothness}
\mathcal V'(a)=-\mathcal V'(b).
\end{equation}
This is the endpoint smoothness condition used in the complete two-dimensional-base analysis of \cite[Proposition~5.7 and Lemma~5.8]{KuhnelRademacher2016}.

From $\mathcal V(a)=\mathcal V(b)=0$ and \eqref{eq:Q-solution},
\[
C_0=\frac{\lambda}{3}a^3-ka
=\frac{\lambda}{3}b^3-kb.
\]
Since $a\ne b$,
\begin{equation}\label{eq:k-root}
k=\frac{\lambda}{3}(a^2+ab+b^2).
\end{equation}
At a zero of $\mathcal V$, \eqref{eq:Q-solution} gives
\[
\mathcal V'(r)=\frac{k}{r}-\lambda r.
\]
Consequently \eqref{eq:endpoint-smoothness} implies
\begin{equation}\label{eq:k-endpoint}
k=\lambda ab.
\end{equation}
If $\lambda\ne0$, comparison of \eqref{eq:k-root} and \eqref{eq:k-endpoint} yields
\[
3ab=a^2+ab+b^2,
\]
so $(a-b)^2=0$, contrary to $a<b$. If $\lambda=0$, the two root equations first give $k=0$ and then $C_0=0$, so \eqref{eq:Q-solution} gives $\mathcal V\equiv0$, again impossible on a nonempty regular set. Thus $r=\psi$ is constant.
\end{proof}

\begin{remark}
The complete warped-product analysis in \cite[Corollary~5.4, Proposition~5.7 and Lemma~5.8]{KuhnelRademacher2016} contains the corresponding ODE and endpoint obstruction. We have retained the two-dimensional calculation because, in the present setting, compactness is first obtained for the base rather than for the lifted total warped product, and the resulting argument is short.
\end{remark}

\begin{theorem}[Compact four-dimensional rigidity]\label{thm:4d-compact}
Let $(M^4,g)$ be a compact Einstein Riemannian manifold carrying a rank-$(2,2)$ conformal product structure with adapted Weyl connection $D$. Then
\begin{equation}
\boxed{D=\nabla^g.}
\end{equation}
In particular, $T_1$ and $T_2$ are $\nabla^g$-parallel.
\end{theorem}

\begin{proof}
By \cref{lem:4d-exact}, $D=\nabla^{e^{2\phi}g}$. \cref{lem:warping-collapse} excludes the nonconstant exact branch, so $\phi$ is constant. Therefore $e^{2\phi}g$ is a constant homothety of $g$ and has the same Levi-Civita connection. Thus $D=\nabla^g$.
\end{proof}

\begin{example}[Positive, zero and negative models]
The theorem has natural models in all three signs of the Einstein constant. For $\lambda>0$, take $S^2_\lambda\times S^2_\lambda$ with the product metric. For $\lambda=0$, a flat four-torus with a parallel $2+2$ splitting gives a compact example. For $\lambda<0$, quotients of $\Hy^2_\lambda\times\Hy^2_\lambda$ by torsion-free cocompact discrete subgroups of $\mathrm{Isom}^+(\Hy^2)\times\mathrm{Isom}^+(\Hy^2)$ preserve the factor distributions and provide negative Einstein examples. Irreducible cocompact lattices give quotients that are not global products of two compact surfaces, so the universal-cover statement below is the correct general classification. Thus the rigidity theorem says that, up to the possible quotient action, these standard product models exhaust the compact rank-$(2,2)$ Einstein case.
\end{example}

\begin{corollary}[Universal-cover classification]\label{cor:universal-cover}
Under the assumptions of \cref{thm:4d-compact},
\begin{equation}\label{eq:universal-classification}
\boxed{
(\widetilde M,\widetilde g)
\cong
(\Sigma_\lambda^2,g_\lambda)
\times
(\Sigma_\lambda^2,g_\lambda),
}
\end{equation}
where $\Sigma_\lambda^2$ is the simply connected complete surface of Gaussian curvature $\lambda$. In particular,
\[
\lambda>0:\quad S^2_\lambda\times S^2_\lambda,
\]
\[
\lambda=0:\quad \R^4,
\]
and
\[
\lambda<0:\quad \Hy^2_\lambda\times\Hy^2_\lambda.
\]
Consequently $g$ is locally symmetric:
\[
\nabla R=0.
\]
\end{corollary}

\begin{proof}
By \cref{thm:4d-compact}, the splitting is $\nabla^g$-parallel. The complete simply connected universal cover therefore splits by de Rham as
\[
(\widetilde M,\widetilde g)=(M_1^2,g_1)\times(M_2^2,g_2).
\]
Since $\Ric_{\widetilde g}=\lambda\widetilde g$, each factor satisfies $\Ric_{g_i}=\lambda g_i$. On a surface $\Ric=Kg$, so both factors have constant Gaussian curvature $\lambda$. This proves \eqref{eq:universal-classification}. Products of space forms are locally symmetric.
\end{proof}

\begin{remark}[Scope of the classification]
The statement of \cref{cor:universal-cover} is intentionally made on the universal cover. In the negative case, compact quotients of $\Hy^2\times\Hy^2$ preserving the two distributions need not be global direct products. No stronger global product assertion is needed for the rigidity result.
\end{remark}

\begin{remark}[Relation with previous rigidity results]
The local closedness statement in \cref{thm:4d-closed} is a consequence of classical ambi-Hermitian geometry. The induced opposite Hermitian structures come from the conformal-product framework of \cite{BelgunMoroianu2011}, while the half-Weyl degeneracy mechanism is standard \cite{ApostolovGauduchon1997,ApostolovCalderbankGauduchon2016}. The curvature interpretation above records how those same conditions appear in the off-diagonal contraction $\KW$ of Section~2.

The new four-dimensional conclusion is global. Theorem~1.1 of \cite{MoroianuPilca2026}, under a specialness assumption, reduces the universal-cover Einstein metric to a warped product. In rank $(2,2)$ we remove that assumption and prove the stronger conclusion $D=\nabla^g$, so the warped product collapses to a Riemannian product. By \cref{cor:4d-KW}, closedness alone does not supply specialness: it leaves the additional term $\theta_1\wedge\theta_2$.

There is also a small but useful distinction in the treatment of the non-exact case. The LCP definition used in \cite[Definition~4.1]{BelgunFlamencourtMoroianu2025} includes non-flatness, so Theorem~4.5 of that paper does not address the flat non-exact similarity branch. We therefore separate that branch and use Fried's classification, rather than absorbing it into the LCP argument. After exactness, the proof uses the full $2+2$ alternative in \cite[Theorem~3.2]{KuhnelRademacher2016} and the two-dimensional warped-product equations in their Section~5.
\end{remark}

\section{Conclusion and an open problem}

For conformal products with $p,q\ge2$, the Faraday form is determined by the off-diagonal part of the $S$-weighted Weyl contraction through \eqref{eq:main-obstruction}. The same curvature calculation gives the Ricci-corrected symmetric identity \eqref{eq:symmetric-companion}. Under block-diagonal Ricci curvature, the Lee-form component formulas of Section~3 impose additional algebraic restrictions on the mixed Weyl term.

The example in \cref{ex:nonclosed-block-Ricci} shows that $[\mathcal K_W,S]$ may be nonzero even when the Ricci tensor is block diagonal. We do not know whether this can occur for an Einstein metric in dimensions greater than four.

\begin{question}[Non-closed Einstein conformal products]
Let $(M^n,g)$ be Einstein and carry a conformal product structure with $p,q\ge2$. Can
\[
[\mathcal K_W,S]\ne0,
\qquad\text{equivalently}\qquad d\theta\ne0,
\]
occur? Equivalently, does the Einstein equation force the adapted Weyl connection to be closed whenever both distribution ranks are at least two? Section~4 answers this question affirmatively in dimension four. In higher dimensions, \eqref{eq:main-obstruction} reduces the problem to the off-diagonal Weyl contraction.
\end{question}

A non-closed Einstein example would also have to be non-special. Indeed, for $p,q\ge2$, the local Einstein analysis of Moroianu--Pilca shows that if either component of the Lee form is closed, then the other component is closed as well; see \cite[Lemma~3.2, Case~1]{MoroianuPilca2024} and the use of that case in the proof of \cite[Theorem~1.1]{MoroianuPilca2026}. Consequently,
\[
d\theta\not\equiv0
\quad\Longrightarrow\quad
\theta_1\ \text{and}\ \theta_2\ \text{are both non-closed}.
\]
The Einstein hypothesis is essential here. In \cref{ex:nonclosed-block-Ricci}, the Ricci tensor is block diagonal but the metric is not Einstein; there $\theta_2$ is closed whereas $\theta_1$ is not. Thus any higher-dimensional Einstein counterexample must lie outside the special class considered in \cite{MoroianuPilca2026}.

There is one further restriction on any possible counterexample. On the open set
\[
\Omega=\{\theta_1\ne0,\ \theta_2\ne0\},
\]
\cref{cor:rank-one} gives $\rk B_W\le1$ for the off-diagonal block $B_W:T_1\to T_2$ of $\mathcal K_W$. Thus an Einstein example with $d\theta\ne0$ would have to satisfy a strong rank-one condition in addition to the Einstein and Bianchi equations.

For later analytic use, the pointwise norm identity also gives, on a compact manifold,
\[
\int_M |d\theta|^2\,dv_g
=
\frac{(n-2)^2}{32(p-1)^2(q-1)^2}
\int_M |[\mathcal K_W,S]|_{\mathrm{HS}}^2\,dv_g.
\]
This is only the integrated form of the pointwise identity, but it may be useful in a Bochner-type study of the higher-dimensional problem.

\appendix
\section{Local mixed Ricci and partial traces}\label{app:mixed-Ricci}

This appendix derives the local Lee forms \eqref{eq:local-Lee}, the two partial curvature traces, and the mixed Ricci formula \eqref{eq:mixed-Ricci-local} directly from the Koszul formula and the curvature convention \eqref{eq:curvature-convention}. The mixed curvature identities are the same local components that appear as equations~(10) and~(13) in \cite{MoroianuPilca2024}. We include the calculation to make the sign conventions transparent. Their traces give a local verification of the partial Weyl formula for $d\theta$ and, under block-diagonal Ricci curvature, of the corresponding partial Riemann trace. Section~2 gives the coordinate-free derivation.

Let
\[
g=e^{2f_1}g_1+e^{2f_2}g_2
\]
on a product chart $\mathcal U_1\times\mathcal U_2$, with $g_i$ a metric on
$\mathcal U_i$ and $f_i\in C^\infty(\mathcal U_1\times\mathcal U_2)$, and set
$T_1=T\mathcal U_1$, $T_2=T\mathcal U_2$.  Throughout, $X,Y,Z$ denote vector
fields obtained from coordinate fields on $\mathcal U_1$ and $U,V,W$ vector
fields obtained from coordinate fields on $\mathcal U_2$; all brackets among
them that we use vanish, and $[X,Y]\in T_1$, $[U,V]\in T_2$.  Write
\[
\zeta_1=(\nabla f_1)_{T_2},
\qquad
\zeta_2=(\nabla f_2)_{T_1}.
\]

\subsection*{A.1. Connection components}

\begin{lemma}
With the above notation,
\begin{align}
\nabla_XU&=U(f_1)X+X(f_2)U,
\label{eq:app-mixed-connection}\\
(\nabla_XY)_{T_2}&=-g(X,Y)\zeta_1,
\qquad
(\nabla_UV)_{T_1}=-g(U,V)\zeta_2.
\label{eq:app-normal-connection}
\end{align}
\end{lemma}

\begin{proof}
The Koszul formula reads
\begin{align*}
2g(\nabla_AB,C)
={}&Ag(B,C)+Bg(A,C)-Cg(A,B)\\
&+g([A,B],C)-g([A,C],B)+g([C,B],A).
\end{align*}
Take $A=X$, $B=U$, $C=Y$.  Then $g(U,Y)=g(X,U)=0$, $[X,U]=[Y,U]=0$ and
$[X,Y]\in T_1$, so all bracket terms drop out and
\[
2g(\nabla_XU,Y)=U\bigl(g(X,Y)\bigr)=2U(f_1)g(X,Y),
\]
because $g(X,Y)=e^{2f_1}g_1(X,Y)$ with $g_1(X,Y)$ independent of the second
factor.  Taking $A=X$, $B=U$, $C=V$ gives, in the same way,
\[
2g(\nabla_XU,V)=X\bigl(g(U,V)\bigr)=2X(f_2)g(U,V).
\]
Together these give \eqref{eq:app-mixed-connection}.  Taking $A=X$, $B=Y$,
$C=U$ yields
\[
2g(\nabla_XY,U)=-U\bigl(g(X,Y)\bigr)=-2U(f_1)g(X,Y)
=-2g(X,Y)\,g(\zeta_1,U),
\]
which is the first identity in \eqref{eq:app-normal-connection}; the second
follows by exchanging the roles of the two factors.
\end{proof}

\begin{lemma}[Local Lee forms]
The adapted Weyl connection of the conformal product determined by
$g=e^{2f_1}g_1+e^{2f_2}g_2$ has components of the Lee form
\begin{equation}\label{eq:app-Lee}
\theta_1=-d_1f_2,
\qquad
\theta_2=-d_2f_1.
\end{equation}
\end{lemma}

\begin{proof}
By \eqref{eq:weyl-connection} and \eqref{eq:app-mixed-connection}, for
$X\in T_1$ and $U\in T_2$,
\[
D_XU=\nabla_XU+\theta(U)X+\theta(X)U-g(X,U)\theta^\sharp
=\bigl(U(f_1)+\theta_2(U)\bigr)X+\bigl(X(f_2)+\theta_1(X)\bigr)U.
\]
Since $D$ preserves $T_2$, the $T_1$-component must vanish, so
$\theta_2(U)=-U(f_1)$ for all $U\in T_2$.  Since $D$ preserves $T_1$, the
same computation applied to
$D_UX=\nabla_UX+\theta(U)X+\theta(X)U$ forces the $T_2$-component to vanish. For the commuting coordinate fields used here, torsion-freeness gives $\nabla_UX=\nabla_XU$, and hence $\theta_1(X)=-X(f_2)$.
\end{proof}

\subsection*{A.2. Mixed curvature and partial traces}

\begin{lemma}
For all $X,Y,Z\in T_1$ and $U\in T_2$,
\begin{equation}\label{eq:app-mixed-curvature}
R(X,Y,Z,U)
=
\Lambda(X,U)\,g(Y,Z)-\Lambda(Y,U)\,g(X,Z),
\end{equation}
where
\[
\Lambda(X,U):=X(f_2)U(f_1)-X(Uf_1).
\]
Symmetrically, for all $U,V,W\in T_2$ and $X\in T_1$,
\begin{equation}\label{eq:app-mixed-curvature-2}
R(U,V,W,X)
=
\mathrm M(X,U)\,g(V,W)-\mathrm M(X,V)\,g(U,W),
\qquad
\mathrm M(X,U):=X(f_2)U(f_1)-X(Uf_2).
\end{equation}
\end{lemma}

\begin{proof}
Both sides of \eqref{eq:app-mixed-curvature} are tensorial, so it suffices to
prove the identity for coordinate fields $X,Y,Z$ on $\mathcal U_1$ and $U$ on
$\mathcal U_2$; then $[X,Y]=0$ and
\[
g(R(X,Y)Z,U)=g(\nabla_X\nabla_YZ,U)-g(\nabla_Y\nabla_XZ,U).
\]
By \eqref{eq:app-normal-connection}, $g(\nabla_YZ,U)=-g(Y,Z)U(f_1)$, and by
\eqref{eq:app-mixed-connection}, $\nabla_XU=U(f_1)X+X(f_2)U$.  Hence
\begin{align*}
g(\nabla_X\nabla_YZ,U)
&=X\bigl(g(\nabla_YZ,U)\bigr)-g(\nabla_YZ,\nabla_XU)\\
&=-X\bigl(g(Y,Z)\bigr)U(f_1)-g(Y,Z)X(Uf_1)\\
&\qquad-U(f_1)g(\nabla_YZ,X)+X(f_2)U(f_1)g(Y,Z).
\end{align*}
Antisymmetrizing in $X$ and $Y$, the terms carrying the factor $-U(f_1)$
combine into
\[
-U(f_1)\Bigl[
X\bigl(g(Y,Z)\bigr)-Y\bigl(g(X,Z)\bigr)
+g(\nabla_YZ,X)-g(\nabla_XZ,Y)
\Bigr]
=-U(f_1)\,g([X,Y],Z)=0,
\]
where we expanded $X(g(Y,Z))=g(\nabla_XY,Z)+g(Y,\nabla_XZ)$ and used
$\nabla_XY-\nabla_YX=[X,Y]=0$.  What remains is exactly
\eqref{eq:app-mixed-curvature}.  Identity
\eqref{eq:app-mixed-curvature-2} follows by exchanging the roles of the two
factors, using $U(Xf_2)=X(Uf_2)$.
\end{proof}

\begin{proposition}[Partial traces]\label{app:partial-traces-formula}
Let $e_1,\ldots,e_p$ and $u_1,\ldots,u_q$ be local $g$-orthonormal frames of
$T_1$ and $T_2$.  Then, for $X\in T_1$ and $U\in T_2$,
\begin{align}
\Ric^{(1)}(X,U):=\sum_{a=1}^{p}R(X,e_a,e_a,U)
&=(p-1)\bigl(X(f_2)U(f_1)-X(Uf_1)\bigr),
\label{eq:app-partial-1}\\
\Ric^{(2)}(X,U):=\sum_{\alpha=1}^{q}R(X,u_\alpha,u_\alpha,U)
&=(q-1)\bigl(X(f_2)U(f_1)-X(Uf_2)\bigr).
\label{eq:app-partial-2}
\end{align}
Consequently
\begin{equation}\label{eq:app-mixed-Ricci}
\boxed{
\Ric(X,U)
=
(1-p)X(Uf_1)+(1-q)X(Uf_2)
+(n-2)X(f_2)U(f_1).
}
\end{equation}
\end{proposition}

\begin{proof}
Put $Y=Z=e_a$ in \eqref{eq:app-mixed-curvature} and sum:
\[
\sum_{a}R(X,e_a,e_a,U)
=\Lambda(X,U)\sum_a g(e_a,e_a)-\sum_a g(X,e_a)\Lambda(e_a,U)
=p\,\Lambda(X,U)-\Lambda(X,U),
\]
because $\Lambda(\cdot,U)$ is a one-form on $T_1$ and
$\sum_ag(X,e_a)e_a=X$.  This is \eqref{eq:app-partial-1}.

For the second trace note that, by the pair symmetry and the two
antisymmetries of $R$,
\[
R(X,u_\alpha,u_\alpha,U)=R(U,u_\alpha,u_\alpha,X).
\]
Putting $V=W=u_\alpha$ in \eqref{eq:app-mixed-curvature-2} and summing gives,
in the same way, $(q-1)\mathrm M(X,U)$, which is \eqref{eq:app-partial-2}.

Finally $\Ric(X,U)=\Ric^{(1)}(X,U)+\Ric^{(2)}(X,U)$ by definition of the
full trace, and adding \eqref{eq:app-partial-1} and
\eqref{eq:app-partial-2} gives \eqref{eq:app-mixed-Ricci}, since
$(p-1)+(q-1)=n-2$.
\end{proof}

This proves \eqref{eq:mixed-Ricci-local} with the sign used throughout the
paper.

\subsection*{A.3. The assumption-free partial Weyl trace}

The local formulas also verify directly that the Ricci contribution cancels from the Weyl partial trace. From the Weyl decomposition, for $X\in T_1$ and $U\in T_2$,
\begin{equation}\label{eq:app-Weyl-correction}
\Ric^{(1)}(X,U)
=
\sum_{a=1}^{p}W(X,e_a,e_a,U)
+\frac{p-1}{n-2}\Ric(X,U).
\end{equation}
Put $u=X(Uf_1)$, $v=X(Uf_2)$, and $w=X(f_2)U(f_1)$. By \cref{app:partial-traces-formula},
\[
\Ric^{(1)}=(p-1)(w-u),
\qquad
\Ric=(n-2)w-(p-1)u-(q-1)v.
\]
Substitution in \eqref{eq:app-Weyl-correction} gives, without imposing $\Ric(X,U)=0$,
\[
\sum_{a=1}^{p}W(X,e_a,e_a,U)
=
\frac{(p-1)(q-1)}{n-2}(v-u).
\]
Finally, \eqref{eq:app-Lee} gives $d\theta(X,U)=v-u$. Hence the mixed partial Weyl trace \eqref{eq:mixed-partial-Weyl} is recovered locally with no curvature hypothesis. This gives an independent check of the Ricci cancellation in \cref{thm:main-obstruction}.

\subsection*{A.4. The partial Riemann traces under block-diagonal Ricci}

Combining the two formulas in \cref{app:partial-traces-formula} under the vanishing of their sum gives the following consequence.

\begin{proposition}\label{app:partial-traces}
If $\Ric(X,U)=0$ for all $X\in T_1$, $U\in T_2$, then
\begin{align}
\Ric^{(1)}(X,U)
&=
\frac{(p-1)(q-1)}{n-2}
\bigl(X(Uf_2)-X(Uf_1)\bigr),
\label{eq:app-partial-solved-1}\\
\Ric^{(2)}(X,U)
&=
-\frac{(p-1)(q-1)}{n-2}
\bigl(X(Uf_2)-X(Uf_1)\bigr).
\label{eq:app-partial-solved-2}
\end{align}
Moreover $d\theta(X,U)=X(Uf_2)-X(Uf_1)$, so
\eqref{eq:partial-Ricci-obstruction} holds.
\end{proposition}

\begin{proof}
Put
\[
u=X(Uf_1),\qquad v=X(Uf_2),\qquad w=X(f_2)U(f_1),
\]
so that \eqref{eq:app-partial-1}--\eqref{eq:app-partial-2} read
\[
\Ric^{(1)}=(p-1)(w-u),
\qquad
\Ric^{(2)}=(q-1)(w-v).
\]
Their sum is $\Ric(X,U)=0$, whence
\[
w=\frac{(p-1)u+(q-1)v}{n-2}.
\]
Substituting,
\[
\Ric^{(1)}
=(p-1)\,\frac{(p-1)u+(q-1)v-(n-2)u}{n-2}
=\frac{(p-1)(q-1)}{n-2}(v-u),
\]
since $(p-1)-(n-2)=-(q-1)$; and $\Ric^{(2)}=-\Ric^{(1)}$.  This gives
\eqref{eq:app-partial-solved-1}--\eqref{eq:app-partial-solved-2}.

For the Faraday form, \eqref{eq:app-Lee} gives $\theta_1=-d_1f_2$ and
$\theta_2=-d_2f_1$, so for commuting coordinate fields $X\in T_1$,
$U\in T_2$,
\[
d\theta_1(X,U)=-U\bigl(\theta_1(X)\bigr)=X(Uf_2),
\qquad
d\theta_2(X,U)=X\bigl(\theta_2(U)\bigr)=-X(Uf_1),
\]
and adding the two gives $d\theta(X,U)=v-u$.
\end{proof}

\end{document}